\documentclass[10pt]{amsart}
\usepackage{egastyle}

\usepackage{amssymb,amsmath,amscd,amsthm,amsfonts,wasysym,mathrsfs,amsxtra,stmaryrd}

\usepackage{graphicx,array,url}
\usepackage[vcentermath]{genyoungtabtikz}
\usepackage{tikz}
\usepackage[section]{placeins}
\usepackage{afterpage}
\usepackage{verbatim}
\input xy
\xyoption{all}

\newcommand{\OO}{\mathcal{O}}

\newcommand{\image}{\textnormal{im}\,}

\newcommand{\Hom}{\textnormal{Hom}}
\newcommand{\dimension}{\textnormal{dim}\,}

\newcommand{\rank}{\textnormal{rk}\,}

\newcommand{\Ac}{\mathcal{A}}
\newcommand{\Bc}{\mathcal{B}}
\newcommand{\Cc}{\mathcal{C}}
\newcommand{\Fc}{\mathcal{F}}
\newcommand{\Tc}{\mathcal{T}}

\newcommand{\Coh}{\mathrm{Coh}}

\newcommand{\arinj}{\ar@{^{(}->}}
\newcommand{\arsurj}{\ar@{->>}}
\newcommand{\areq}{\ar@{=}}

\newcommand{\wh}{\widehat}

\newcommand{\Bl}{\mathcal{B}^l}
\newcommand{\ch}{\mathrm{ch}}
\newcommand{\oZw}{\overline{Z}_\omega}
\newcommand{\oZl}{\overline{Z}^l}

\newcommand{\scalea}{\scalebox{0.5}}

\newcommand{\whPhi}{{\wh{\Phi}}}
\newcommand{\bo}{{\overline{\omega}}}
\newcommand{\wt}{\widetilde}
\newcommand{\ol}{\overline}
\newcommand{\ysize}{\Yboxdim{9pt}}

\makeatletter
\newtheorem*{rep@theorem}{\rep@title}
\newcommand{\newreptheorem}[2]{%
\newenvironment{rep#1}[1]{%
 \def\rep@title{#2 \ref{##1}}%
 \begin{rep@theorem}}%
 {\end{rep@theorem}}}
\makeatother

\newreptheorem{theorem}{Theorem}

\begin{document}

\title[Fourier-Mukai transforms of  stable sheaves on Weierstrass elliptic threefolds]{Fourier-Mukai transforms of slope stable torsion-free sheaves on Weierstrass elliptic threefolds}

\author[Jason Lo]{Jason Lo}
\address{Department of Mathematics \\
California State University, Northridge\\
18111 Nordhoff Street\\
Northridge CA 91330 \\
USA}
\email{jason.lo@csun.edu}


\begin{abstract}
We focus on a class of  Weierstra{\ss} elliptic threefolds  that allows the base of the fibration to be a Fano surface or a numerically $K$-trivial surface.  We define the notion of limit tilt stability, and   show that the Fourier-Mukai transform of a slope stable torsion-free sheaf satisfying a vanishing condition in codimension 2 (e.g.\ a reflexive sheaf) is a limit tilt stable object.  We also show that   the inverse Fourier-Mukai transform of a limit tilt semistable object of nonzero fiber degree is a slope semistable torsion-free sheaf, up to modification in codimension 2.
\end{abstract}

\maketitle
\tableofcontents

\section{Introduction}

\subsection{Background and motivation} Given the derived category of coherent sheaves $D^b(X)$ on a smooth projective variety $X$ and an autoequivalence $\Phi$ of $D^b(X)$, the natural and fundamental question of how $\Phi$ affects slope stability or Gieseker stability for sheaves on $X$ has been studied in a multitude of works.  On threefolds,  there is Friedman-Morgan-Witten's spectral cover construction of stable sheaves on elliptic threefolds, which involves applying Fourier-Mukai transforms (FMTs) to stable sheaves supported on hypersurfaces (the `spectral covers'), which are rank-zero stable sheaves \cite{FMW}.  In Bridgeland-Maciocia's work \cite{BMef}, where they established the existence of Fourier-Mukai transforms on Calabi-Yau fibrations of relative dimension at most two, they  showed  that any connected component of the moduli of rank-one torsion-free sheaves on an elliptic threefold is isomorphic to a moduli of higher-rank Gieseker stable sheaves via a Fourier-Mukai transform.

In light of the connection between  the Bridgeland stability manifold  and  mirror symmetry \cite{SCK3,StabTC}, it has also become important to understand the action of autoequivalences on Bridgeland stability conditions.  Recently, the action of autoequivalences on specific moduli spaces of stable objects (e.g. stable pairs in the sense of Pandharipande-Thomas) has  proved relevant to understanding the structures of Pandharipande-Thomas (PT) invariants and Donaldson-Thomas (DT) invariants on elliptic threefolds, including their modularity  \cite{OS1,Dia15}.  Since slope stable torsion-free sheaves are basic building blocks of other stable objects such as Bridgeland stable objects and PT stable pairs, it is only natural to  study the Fourier-Mukai transforms of slope stable torsion-free sheaves on elliptic fibrations.

More specifically, Bayer-Macr\`{i}-Toda's construction of Bridgeland stability conditions on threefolds \cite{BMT1} (which has been verified on numerous types of threefolds, e.g.\ see \cite{li2019stability}) begins with a tilt of the category of coherent sheaves using slope stability for sheaves.  This means that in principle, in order to understand the action of the autoequivalence on the Bridgeland stability manifold, we need to first understand the autoequivalence group action on  slope stability for sheaves.  This is the primary motivation for Theorem \ref{thm:intro1}, our main result in this article.

  The elliptic surface analogue of Theorem \ref{thm:intro1} was established in the author's joint work with Liu and Martinez in \cite{LLM}, the results in which were then used to describe the actions of autoequivalences on Bridgeland stability conditions on elliptic surfaces \cite{Lo20}.

\subsection{Main results}  In this article, we focus on elliptic threefolds.  In particular, we consider  a class of  elliptic threefolds $p : X \to B$ that are Weierstra{\ss} and where $X$ is $K$-trivial.  Such threefolds admit  a Fourier-Mukai transform $\Phi : D^b(X) \overset{\thicksim}{\to} D^b(X)$ given by the relative Poincar\'{e} sheaf.

To state our main results, we need to recall that in Bayer-Macr\`{i}-Toda's conjectural construction of Bridgeland stability conditions on  threefolds, one needs an intermediary between slope stability for sheaves and the conjectured Bridgeland stability.  This intermediary is called tilt stability, and it is a notion of stability for certain 2-term complexes (see \ref{para:slopetiltstab} for a precise definition).  The definition of tilt stability involves an ample class $\omega$ on the threefold.  On an elliptic threefold $X$, if we deform $\omega$ along a hyperbola in the ample cone of $X$ so that $\omega$ approaches the  `fiber direction', we obtain a notion of stability  which we call   \emph{limit tilt stability}.  Limit tilt stability is similar to Bayer's notion of  polynomial stability condition.  Instead of having central charges that are polynomials, however,  we allow the central charge of limit tilt stability to be power series (see \ref{para:oZlstabasymppolystab}).

 %Our main technical result can be paraphrased as follows:

\begin{thm}[Theorem \ref{thm:paper14thm6-7analogue}]\label{thm:intro0}
Limit tilt stability has the Harder-Narasimhan  property.
\end{thm}

The Harder-Narasimhan (HN) property is a property expected of any reasonable notion of stability.  Theorem \ref{thm:intro0} ensures that standard properties that we have come to expect of stable objects, such as Hom-vanishing from a stable object of a larger phase to a stable object of a smaller phase, also apply to limit tilt stable objects.

Our   main theorem can  be paraphrased as follows:

\begin{thm}[Theorem \ref{thm:paper14thm5-1analogue}]\label{thm:intro1}
Let $X$ be a $K$-trivial Weierstra{\ss} elliptic threefold.
\begin{itemize}
\item[(A)] If $E$ is a slope stable torsion-free sheaf on $X$ such that certain morphisms from 1-dimensional sheaves into $E$ always vanish,  then $E$ is taken by $\Phi$ to a limit tilt stable object in $D^b(X)$.
\item[(B)] If $F$  is a  limit tilt semistable object in $D^b(X)$, then it can be modified in codimension 2 to $F'$ so that  the quasi-inverse of $\Phi$ takes $F'$ to a slope semistable torsion-free sheaf on $X$.
\end{itemize}
\end{thm}

The   vanishing condition in Theorem \ref{thm:intro1}(A) is always satisfied by a torsion-free reflexive sheaf (see Corollary \ref{cor:thm1reflexiveexample}).  Since a torsion-free sheaf $F$ differs from its double dual $F^{\ast \ast}$ (which is reflexive) only in codimension 2, we can also think of the first part of Theorem \ref{thm:paper14thm5-1analogue} as saying that every slope stable torsion-free sheaf is taken to a limit tilt stable object up to modification in codimension 2.

Informally, we can visualise the relation between slope stability for  sheaves and limit tilt stability in the diagram
\[
\xymatrix{
  \text{slope stable  (with Hom vanishing)} \ar@{=>}[r]^(.64){\text{FMT}} & \text{limit tilt stable} \ar@{=>}[d] \\
  \text{slope semistable} & \ar@{=>}[l]_{\text{inverse FMT}}^{\text{mod.\ in codim.\ 2}} \text{limit tilt semistable}
}
\]
which is reminiscent of the relation between slope stability and Gieseker stability for sheaves:
\[
\xymatrix{
  \text{slope stable} \ar@{=>}[r]  & \text{Gieseker stable} \ar@{=>}[d] \\
  \text{slope semistable} & \ar@{=>}[l] \text{Gieseker semistable}
}
\]
Therefore, one can  think of limit tilt stability as a ``refinement'' of slope stability up to the autoequivalence $\Phi$, just as Gieseker stability is a refinement of slope stability.  

\textbf{Main ideas of proofs.} Theorem \ref{thm:intro1} is a result that compares two notions of stability up to an autoequivalence: slope stability for sheaves, and limit tilt stability for complexes of sheaves.  Each of these two notions of stability is constructed by specifying  the heart of a t-structure and a central charge: In the case of slope stability, the heart is the category of coherent sheaves while the central charge is given by the rank and degree functions (see \ref{para:slopetiltstab}).  In the case of limit tilt stability, the heart $\Bc^l$ is constructed as a `limit' of the heart Bayer-Macr\`{i}-Toda constructed for tilt stability (Lemma \ref{lem:paper14lem4-6analogue}), while the central charge $\ol{Z}_\omega$ has the same form as tilt stability, but considered as a function that takes values in power series over $\mathbb{C}$ rather than just  in $\mathbb{C}$.  As a result, the proof of Theorem \ref{thm:intro1} is necessarily a comparison (up to an autoequivalence) of the hearts and the central charges of slope stability and limit tilt stability.  The details of this comparison are carried out in stages,  in Sections \ref{sec:cohomFMT} through \ref{sec:slopestabvslimittilstab}.

An innovation in the proofs of this article is finding an appropriate curve in the ample cone such that, when we deform an ample class $\omega$ along this curve, we obtain a reasonable notion of stability (limit tilt stability) that roughly matches  slope stability up to an autoequivalence.  The curve we found turns out to be a hyperbola (see \eqref{eq:limitcurve}).  This is in contrast to the large volume limit, which is a polynomial stability condition  obtained by deforming $\omega$ along a ray in the ample cone \cite[Section 6]{BMT1}.

The proof of Theorem \ref{thm:intro0}, on the other hand, consists of finding an `approximation' of the HN filtration first, followed by refining it to the HN filtration.  The approximation is given by a  torsion quadruple (see \ref{eq:paper14eq35}), which gives a 4-step filtration of any object in the derived category.  To refine this to the HN filtration, we adapt Toda's  argument for  refining a torsion pair (a 2-step filtration) into the HN filtration of a polynomial stability condition \cite{Toda1}.

\subsection{Relation to prior work} In joint work with Zhang \cite{LZ2},  the author considered the product threefold $X = C \times B$ where $C$ is a smooth elliptic curve and $B$ is a K3 surface of Picard rank 1.  The second projection $p : X \to B$ is  an elliptic fibration where all the fibers are copies of $C$, and hence are smooth, while  the Fourier-Mukai transform $D^b(C) \overset{\thicksim}{\to} D^b(C)$ given by the Poincar\'{e} line bundle can be lifted to a Fourier-Mukai transform $\Phi : D^b(X) \overset{\thicksim}{\to} D^b(X)$.  In  the predecessor  \cite{Lo14} to this article, the author proved a  coarser version of Theorem  \ref{thm:paper14thm5-1analogue} for the product threefold $X = C \times B$.  In the present article, more care is needed in dealing with technical issues such as a section of the fibration not being nef, or that there are nontrivial contributions from the base of the fibration towards cohomological Fourier-Mukai transforms.  Therefore, even though the results in the preceding article \cite{Lo14}  were less general, because of the simpler notation in the product case, the exposition there provides a more illuminating picture of the key ideas in this article.

\subsection{Outline of the paper}  In Section \ref{sec:prelim}, we lay out the assumptions we place on our elliptic threefold $p : X \to B$, and review preliminary concepts that will be used throughout the paper.  In Section \ref{sec:cohomFMT}, we explain how a matrix notation for Chern characters makes it easier to understand the cohomological Fourier-Mukai transform $\Phi$; we also solve a numerical equivalence that is necessary for comparing slope stability with tilt stability up to the autoequivalence $\Phi$.  In Section \ref{sec:limittiltstabconstruction}, we give the construction of limit tilt stability.  Section \ref{sec:phasecomputations} includes  computations of  phases of various objects with respect to limit tilt stability.  In Section \ref{sec:slopestabvslimittilstab}, we prove the main theorem, Theorem \ref{thm:paper14thm5-1analogue}, which is a comparison theorem between slope stability and limit tilt stability via the Fourier-Mukai transform $\Phi$.  We end the  paper with Section \ref{sec:HNpropertylimittiltstab}, in which we verify   the Harder-Narasimhan property of limit tilt stability.

\subsection{Acknowledgements}  The author thanks Wanmin Liu for  sharing his knowledge on elliptic fibrations through many helpful discussions, Ziyu Zhang for generously sharing his insight on  cohomological Fourier-Mukai transforms on elliptic fibrations, and Ching-Jui Lai for answering the author's  questions on elliptic fibrations.  He would also like to thank Antony Maciocia for a suggestion raised during the author's talk at Edinburgh in May 2017, which led to an improvement in the proofs of HN properties in this article.  In addition, he thanks Conan Leung, Jun Li and  Arend Bayer for conversations that provided the impetus for completing this work.  He thanks the Institute for Basic Science in Pohang, South Korea, for their support throughout his visit in June-July 2017, during which a substantial portion of this work was completed.  The author also thanks the referee for comments that led to substantial improvement of the article.  He was partially supported by NSF Grant DMS-2100906 and DARPA grant D21AP10109-00 during the final stage of  this project.

\section{Preliminaries}\label{sec:prelim}

\paragraph[Our elliptic fibration]  Throughout \label{para:ellipfibdef} this article, all the schemes will be over $\mathbb{C}$.  We will write $p : X \to B$ to denote a Weierstra{\ss}  threefold in the sense of \cite[Section 6.2]{FMNT} where $X$ and $B$ are both smooth projective varieties.  In particular, this means that $p$ is an elliptic threefold that is  a Weierstra{\ss} fibration.  By $p$ being an elliptic threefold, we mean that $X$ is a threefold and $B$ is a surface, while $p$ is a flat morphism whose fibers are Gorenstein curves of arithmetic genus 1; by $p$ also being Weierstra{\ss}, we mean that all the fibers of $p$ are geometrically integral and there exists a section $\sigma : B \hookrightarrow X$ such that the image  $\Theta = \sigma (B)$ does not intersect any singular point of the singular fibers.  The smoothness of $X$ implies that the generic fiber of $p$ is a smooth elliptic curve.  The existence of a section ensures that the singular fibers of $p$ can only be nodal or cuspidal curves.

We will assume the following on our elliptic fibration $p$ from Section \ref{sec:limittiltstabconstruction} through the end of the article:
\begin{enumerate}
  \item[(1)] $X$ is $K$-trivial, i.e.\ $\omega_X \cong \OO_X$.
  \item[(2)] The  cohomology ring over $\mathbb{Q}$ has a decomposition
  \begin{equation}\label{eq:assumptioncohomring}
    H^{2i}(X,\mathbb{Q}) = \Theta p^\ast H^{2i-2}(B,\mathbb{Q}) \oplus p^\ast H^{2i}(B,\mathbb{Q}).
  \end{equation}
  \item[(3)] There exists an ample class $H_B$ on $B$ satisfying:
  \begin{itemize}
  \item[(3a)] There exists $v' >0$ such that, for all $v > v'$, the divisor $\Theta + vp^\ast H_B$ is ample on $X$.
  \item[(3b)] $K_B \equiv hH_B$ for some  $h \in \mathbb{R}$, i.e.\ the canonical divisor of $B$ is numerically equivalent to some  real (possibly zero) multiple  of $H_B$.
  \end{itemize}
\end{enumerate}

We do not expect assumption (1) to be essential.  Assumption (2) is needed for computing intersection numbers; in principle, however,  it can be replaced by any explicit description of the structure of the cohomology ring of $X$.  Assumption (3) is the most crucial among the three, as it is needed for  comparing   stability  `before'   and  stability `after' applying the Fourier-Mukai transform.

We do not impose the vanishing $H^1(X,\OO_X)=0$ as in the case of \cite[Section 6.2.6]{FMNT} or \cite[Section 5]{MV}.  In particular, we allow the base $B$ to be  a K3 surface.

\smallskip

\textbf{Example.} Conditions (1) and (2) are assumed for the Calabi-Yau threefolds considered in \cite[Section 6.6.3]{FMNT}, where  only four possibilities of $B$ are considered: $B$ must be a Fano (i.e.\ del Pezzo) surface, a Hirzebruch surface, an Enriques surface, or the blowup of a Hirzebruch surface.  Condition (3a) is satisfied for all these four possibilities of $B$ (see \cite[Section 6.6.3]{FMNT}), while condition (3b) is satisfied when $B$ is Fano (in which case $H_B$ can be chosen to be $-K_B$ and  $h=-1$) or Enriques (in which case $K_B \equiv 0$ and $h=0$).

\textbf{Example.} Conditions (1), (2) and (3) are all satisfied for the elliptic threefolds considered by Diaconescu in \cite{Dia15}.  In the article \cite{Dia15}, it is assumed that  (1) holds and that the base  $B$ of the fibration  is Fano, and so condition (3b) holds if we choose $H_B=-K_B$ and $h=-1$.  A  decomposition of the form (2) is proved for the torsion-free part of the cohomology ring over $\mathbb{Z}$ in  \cite[Lemma 2.1]{Dia15}; by considering the cohomology ring over $\mathbb{Q}$, the torsion elements vanish and so condition (2)  holds.  By \cite[Corollary 2.2(i)]{Dia15}, if we choose $H_B=-K_B$ then  condition (3a) holds for $v'=1$.

\textbf{Example.} Conditions (1), (2) and (3) are all satisfied when $X = C \times S$ is the product of a smooth elliptic curve $C$ and a K3 surface $S$, and the fibration $p : X \to S$ is the second projection (i.e.\ $p$ is a trivial fibration) as in the case of the author's previous work \cite{Lo14, LZ2}.  In this case, condition (1) clearly holds,  condition (2) follows from \cite[Section 4.2]{LZ2},  condition (3a) follows from \cite[Lemma 4.7]{LZ2} by choosing $v'=0$, while condition (3b) holds by taking $h=0$ since $K_B = \mathcal{O}_B$.

\paragraph[Notation] We collect here some notions and notations that will be used throughout the article.

\subparagraph[Twisted Chern character] For any divisor $B$ on a smooth projective threefold $X$ and any $E \in D^b(X)$, the twisted Chern character $\ch^B(E)$ is defined as
\[
\ch^B(E) = e^{-B}\ch(E)
= (1-B+\tfrac{B^2}{2}-\tfrac{B^3}{6})\ch(E).
\]
  We write $\ch^B(E) = \sum_{i=0}^3 \ch_i^B(E)$ where
\begin{align*}
  \ch_0^B(E) &= \ch_0(E) \\
  \ch_1^B(E) &= \ch_1(E)-B\ch_0(E) \\
  \ch_2^B(E) &= \ch_2(E) -B\ch_1(E)+\tfrac{B^2}{2}\ch_0(E) \\
  \ch_3^B(E) &= \ch_3(E) -B\ch_2(E) + \tfrac{B^2}{2}\ch_1(E) - \tfrac{B^3}{6}\ch_0(E).
\end{align*}
We refer to $B$ as the `$B$-field' involved in the twisting of the Chern character.  In this article, there should be no risk of confusion as to whether  `$B$' refers to a $B$-field or the base of our elliptic threefold as in \ref{para:ellipfibdef}, as this will always be clear from the context.

\subparagraph Suppose $\Ac$ is an abelian category and $\Bc$ is the heart of a t-structure on $D^b(\Ac)$.  For any object $E \in D^b(\Ac)$, we will write $\mathcal{H}^i_{\Bc}(E)$ to denote the $i$-th cohomology object of $E$ with respect to the t-structure with heart $\Bc$.  When $\Bc = \Ac$, i.e.\ when the aforementioned t-structure is the standard t-structure on $D^b(\Ac)$, we will write $H^i(E)$ instead of $\mathcal{H}^i_{\Ac}(E)$.

Given a smooth projective variety $X$, the dimension of an object $E \in D^b(X)$ will be denoted by $\dimension E$, and refers to the dimension of its support, i.e.\
\[
  \dimension E = \dimension \bigcup_i \mathrm{supp}\, H^i(E).
\]
That is, for a coherent sheaf $E$, we have $\dimension E = \dimension \mathrm{supp} (E)$.

\subparagraph[Torsion pairs and tilting]  A \label{para:torsionpairtilting} torsion pair $(\Tc, \Fc)$ in an abelian category $\mathcal{A}$ is a pair of full subcategories $\Tc, \Fc$ such that
\begin{itemize}
\item[(i)] $\Hom_{\Ac}(E', E'')=0$ for all $E' \in \Tc, E'' \in \Fc$.
\item[(ii)] Every object $E \in \Ac$ fits in an $\Ac$-short exact sequence
\[
0 \to E' \to E \to E'' \to 0
\]
for some $E' \in \Tc, E'' \in \Fc$.
\end{itemize}
The decomposition of $E$ in (ii) is canonical \cite[Chapter 1]{HRS}, and we will occasionally refer to it as the $(\Tc,\Fc)$-decomposition of $E$ in $\Ac$.

Whenever we have a torsion pair $(\Tc, \Fc)$ in an abelian category $\mathcal{A}$, we will refer to $\Tc$ (resp.\ $\Fc$) as the torsion class (resp.\ torsion-free class) of the torsion pair.  The extension closure in $D^b(\Ac)$
\[
  \Ac' = \langle \Fc [1], \Tc \rangle
\]
is the heart of a t-structure on $D^b(\Ac)$, and hence an abelian subcategory of $D^b(\Ac)$.  We call $\Ac'$ the tilt of $\Ac$ at the torsion pair $(\Tc, \Fc)$.  More specifically, the category $\Ac'$ is the heart of the t-structure $(D^{\leq 0}_{\Ac'}, D^{\geq 0}_{\Ac'})$ on $D^b(\Ac)$ where
\begin{align*}
  D^{\leq 0}_{\Ac'} &= \{ E \in D^b(\Ac) : \mathcal{H}_{\Ac}^0 (E)\in \Tc, \mathcal{H}^i_{\Ac} (E)= 0 \, \forall\, i > 0 \}, \\
  D^{\geq 0}_{\Ac'} &= \{ E \in D^b(\Ac) : \mathcal{H}_{\Ac}^{-1} (E)\in \Fc, \mathcal{H}^i_{\Ac} (E)= 0 \,\forall\, i < -1 \}.
\end{align*}

A subcategory of $\Ac$ will be called a torsion class (resp.\ torsion-free class) if it is the torsion class (resp.\ torsion-free class) in some torsion pair in $\Ac$.  By a lemma of Polishchuk \cite[Lemma 1.1.3]{Pol}, if $\Ac$ is a noetherian abelian category, then every subcategory that is closed under extension and quotient in $\Ac$ is a torsion class in $\Ac$.  The reader may  refer to
\cite{HRS} for the basic properties of torsion pairs and tilting.

For any subcategory $\mathcal{C}$ of an abelian category $\mathcal{A}$, we will set
\[
  \mathcal{C}^\circ = \{ E \in \mathcal{A} : \Hom_{\mathcal{A}}(F,E)=0 \text{ for all } F \in \mathcal{C} \}
\]
when $\mathcal{A}$ is clear from the context.  Note that whenever $\mathcal{A}$ is noetherian and $\mathcal{C}$ is closed under extension and quotient in $\mathcal{A}$, the pair $(\mathcal{C},\mathcal{C}^\circ)$ gives a torsion pair in $\mathcal{A}$.

\subparagraph[Torsion $n$-tuples] A torsion $n$-tuple $(\Cc_1, \Cc_2,\cdots,\Cc_n)$ in an abelian category $\Ac$ as defined in \cite[Section 2.2]{Pol2} is a collection of full subcategories of $\Ac$ such that
\begin{itemize}
\item $\Hom_{\Ac} (C_i,C_j)=0$ for any $C_i \in \mathcal C_i, C_j \in \mathcal C_j$ where $i<j$.
\item Every object $E$ of $\Ac$ admits a filtration in $\Ac$
\[
  0=E_0 \subseteq E_1 \subseteq E_2 \subseteq \cdots \subseteq E_n = E
\]
where $E_i/E_{i-1} \in \mathcal C_i$ for each $1 \leq i \leq n$.
\end{itemize}
The same notion  also appeared in Toda's work \cite[Definition 3.5]{Toda2}.  Note that, given a torsion $n$-tuple in $\Ac$ as above, the pair $(\langle \Cc_1, \cdots, \Cc_i\rangle, \langle \Cc_{i+1}, \cdots, \Cc_n \rangle )$ is a torsion pair in $\Ac$ for any $1 \leq i \leq n-1$.

\subparagraph[Fourier-Mukai transforms] For any Weierstra{\ss} elliptic fibration $p : X \to B$ in the sense of \cite[Section 6.2]{FMNT} where $X$ is smooth (which implies that $B$ is Cohen-Macaulay \cite[Proposition C.1]{FMNT}), there is a pair of relative Fourier-Mukai transforms $\Phi, \whPhi : D^b(X) \overset{\thicksim}{\to} D^b(X)$ whose kernels are both sheaves on $X \times _B X$, satisfying
\begin{equation}\label{eq:PhiwhPhiisidshifted}
  \whPhi \Phi = \mathrm{id}_{D^b(X)}[-1] = \Phi \whPhi.
\end{equation}
In particular, the kernel of $\Phi$ is the relative Poincar\'{e} sheaf for the fibration $p$, which is a universal sheaf for the moduli problem that parametrises degree-zero, rank-one torsion-free sheaves on the fibers of $p$.  An object $E \in D^b(X)$ is said to be $\Phi$-WIT$_i$ if $\Phi E$ is a coherent sheaf sitting at degree $i$.  In this case, we write $\wh{E}$ to denote the coherent sheaf satisfying  $\Phi E \cong \wh{E} [-i]$ up to isomorphism.  The notion of $\whPhi$-WIT$_i$ can similarly be defined.  The identities \eqref{eq:PhiwhPhiisidshifted} imply that, if a coherent sheaf $E$ on $X$ is $\Phi$-WIT$_i$ for $i=0, 1$, then $\wh{E}$ is $\whPhi$-WIT$_{1-i}$.  For $i=0,1$, we will define the category
\[
  W_{i,\Phi} = \{ E \in \Coh (X) : E \text{ is $\Phi$-WIT$_i$} \}
\]
and similarly for $\whPhi$.  Due to the symmetry between $\Phi$ and $\whPhi$, the properties held by $\Phi$ also hold for $\whPhi$.     The reader may refer to \cite[Section 6.2]{FMNT} for more background on the functors $\Phi, \whPhi$.

\subparagraph[Subcategories of $\Coh (X)$] We fix our notation for various full subcategories of $\Coh (X)$ that will be used throughout this paper.  For any integers $d \geq e$, we set
\begin{align*}
\Coh^{\leq d}(X) &= \{ E \in \Coh (X): \dimension \mathrm{supp}(E) \leq d\} \\
\Coh^{= d}(X) &= \{ E \in \Coh (X): \dimension \mathrm{supp}(E) = d\} \\
\Coh^{\geq d}(X) &= \{ E \in \Coh (X): \Hom_{\Coh (X)} (F,E)=0 \text{ for all }F \in \Coh^{\leq d-1}(X)\} \\
\Coh (p)_{\leq d} &= \{ E \in \Coh (X): \dimension p (\mathrm{supp}(E)) \leq d\} \\
\Coh^d(p)_e &= \{ E \in \Coh (X): \dimension \mathrm{supp}(E) = d, \dimension p (\mathrm{supp}(E)) = e\} \\
\Coh (p)_0 &= \{ E \in \Coh (X) : \dimension p (\mathrm{supp}(E)) = 0\} \\
\{ \Coh^{\leq 0} \}^\uparrow &= \{ E \in \Coh (X): E|_b \in \Coh^{\leq 0} (X_b) \text{ for all closed points $b \in B$} \}
\end{align*}
where $\Coh^{\leq 0}(X_b)$ is the category of coherent sheaves supported in dimension 0 on the fiber $p^{-1}(b)=X_b$, for the closed point $b \in B$.  We will refer to coherent sheaves that are supported on a finite number of fibers of $p$ as fiber sheaves; then  $\Coh (p)_0$ is precisely the category of fiber sheaves on $X$.

\subparagraph[Torsion classes in $\Coh (X)$] For any integer $d$, the categories $\Coh^{\leq d}(X), \Coh (p)_{\leq d}, \Coh (p)_0$ as well as $\{\Coh^{\leq 0}\}^\uparrow, W_{0,\whPhi}$ are all torsion classes in $\Coh (X)$.  From \ref{para:torsionpairtilting}, each of these torsion classes determines a tilt of $\Coh (X)$, and hence determines a t-structure on $D^b(X)$. The behaviours of these t-structures under the Fourier-Mukai transforms $\Phi, \whPhi$ were studied by Zhang and the author  in \cite{Lo7, Lo11, LZ2}.  In particular, we have the torsion pairs $(W_{0,\whPhi}, W_{1,\whPhi})$ and $(\Coh^{\leq d}(X), \Coh^{\geq d+1}(X))$ in $\Coh (X)$.

\subparagraph[Slope stability and tilt stability]\label{para:slopetiltstab} Suppose \label{para:muomegaBslopefctndefn} $X$ is a smooth projective threefold with a fixed ample divisor   $\omega$ and a fixed divisor $B$.  For any coherent sheaf $E$ on $X$, we  define
\[
  \mu_{\omega,B} (E) = \begin{cases}
  \frac{\omega^2 \ch_1^B (E)}{\ch_0^B(E)} &\text{ if $\ch_0^B(E) \neq 0$} \\
  +\infty &\text{ if $\ch_0^B(E)=0$}\end{cases}.
\]
A coherent sheaf $E$ on $X$ is said to be $\mu_{\omega,B}$-stable or slope stable (resp.\ $\mu_{\omega,B}$-semistable or slope semistable) if, for every short exact sequence in $\Coh (X)$ of the form
\[
0 \to M \to E \to N \to 0
\]
where $M,N \neq 0$, we have $\mu_{\omega,B} (M) < (\text{resp.}\, \leq) \, \mu_{\omega,B} (N)$.  The slope function $\mu_{\omega,B}$ has the Harder-Narasimhan (HN) property, i.e. every coherent sheaf $E$ on $X$ has a filtration by coherent sheaves
\[
  E_0 \subsetneq E_1 \subsetneq \cdots \subsetneq E_n = E
\]
where each $E_i/E_{i-1}$ is $\mu_{\omega,B}$-semistable and $\mu_{\omega,B} (E_0) > \mu_{\omega,B} (E_1/E_0) > \cdots > \mu_{\omega,B} (E_n/E_{n-1})$.

For any  coherent sheaf $M$ with $\ch_0(M) \neq 0$,  we have
\[
\mu_{\omega,B} (M) = \frac{\omega^2 \ch_1^B(M)}{\ch_0(M)} = \frac{\omega^2\ch_1(M)-\omega^2 B\ch_0(M)}{\ch_0(M)} = \mu_\omega (M) -\omega^2B.
\]
Hence $\mu_{\omega,B}$-stability is equivalent to $\mu_\omega$-stability for  coherent sheaves.

The HN property of the slope function $\mu_{\omega,B}$ on $\Coh (X)$ implies  that the subcategories of $\Coh (X)$
\begin{align*}
  \Tc_{\omega,B} &= \langle E \in \Coh (X) : E \text{ is $\mu_{\omega,B}$-semistable}, \mu_{\omega,B} (E)> 0 \rangle, \\
  \Fc_{\omega,B} &= \langle E \in \Coh (X) : E \text{ is $\mu_{\omega,B}$-semistable}, \mu_{\omega,B} (E) \leq 0 \rangle
\end{align*}
form a torsion pair in $\Coh (X)$ and  give rise to the tilt of $\Coh (X)$
\[
  \Bc_{\omega,B} = \langle \Fc_{\omega,B} [1], \Tc_{\omega,B} \rangle.
\]
For any object $F \in \Bc_{\omega,B}$, we  set
\[
 \nu_{\omega,B} (F) = \begin{cases}
 \frac{\omega\ch_2^B(F) - \tfrac{\omega^3}{6}\ch_0^B(F)}{\omega^2 \ch_1^B(F)} &\text{ if $\omega^2 \ch_1^B(F) \neq 0$} \\
 +\infty &\text{ if $\omega^2 \ch_1^B(F) = 0$}\end{cases}.
\]
An object $F \in \Bc_{\omega,B}$ is said to be $\nu_{\omega,B}$-stable or tilt stable (resp.\ $\nu_{\omega,B}$-semistable or tilt semistable) if, for every short exact sequence in $\Bc_{\omega,B}$
\begin{equation}\label{eq:Bcomegases}
0 \to M \to F \to N \to 0
\end{equation}
where $M,N \neq 0$, we have $\nu_{\omega,B} (M) < (\text{resp.} \leq )\, \nu_{\omega,B} (N)$.  Tilt stability is a key notion in a now-standard construction of Bridgeland stability conditions on threefolds \cite{BMT1}.

When $B=0$, we  often drop the subscript $B$ in  $\mu_{\omega, B}, \nu_{\omega, B}$ and $\Tc_{\omega,B}, \Fc_{\omega,B},\Bc_{\omega,B}$ above; for instance, we simply write $\mu_\omega$ instead of $\mu_{\omega, 0}$.

\subparagraph[(Reduced) central charge]\label{para:tiltstab2} We will write
\[
\mathbb{H} = \{ re^{i\theta \pi} : r>0, \theta \in (0,1] \}
\]
to denote the strict upper-half complex plane together with the negative real axis, and write
\[
\mathbb{H}_0 = \mathbb{H} \cup \{0\}.
\]
When $X$ is a smooth projective threefold and $\omega$ is an ample divisor on $X$, we  define the reduced central charge $\ol{Z}_{\omega} : K(D^b(X)) \to \mathbb{C}$ where
\begin{equation*}
 \ol{Z}_{\omega} (F) = \tfrac{1}{2}\omega^2\ch_1(F) + i \left( \omega \ch_2(F) - \tfrac{\omega^3}{6} \ch_0(F)\right).
\end{equation*}
Then for any nonzero $F \in \Bc_\omega$, we have $\ol{Z}_\omega (F) \in -i\mathbb{H}_0$ \cite[Remark 3.3.1]{BMT1}.  We can also define the phase $\phi (F)$ of a nonzero object $F$ in $\Bc_\omega$ by setting $\phi (F) = \tfrac{1}{2}$ if $\ol{Z}_\omega (F)=0$, and requiring
\[
  \ol{Z}_\omega (F) \in \mathbb{R}_{>0}e^{i\phi (F) \pi} \text{ where $\phi (F) \in (-\tfrac{1}{2},\tfrac{1}{2}]$}
\]
if $\ol{Z}_\omega (F) \neq 0$.  An object $F \in \Bc_\omega$ is then said to be $\ol{Z}_\omega$-stable (resp.\ $\ol{Z}_\omega$-semistable) if, for every short exact sequence \eqref{eq:Bcomegases} in $\Bc_\omega$, we have $\phi (M) < (\text{resp.} \leq )\,  \phi (N)$.  Note that $\ol{Z}_\omega$-stability is equivalent to $\nu_\omega$-stability, i.e.\ tilt stability.

\subparagraph[Slope-like functions]  Suppose $\Ac$ is an abelian category.  We call a function $\mu$ on $\Ac$ a slope-like function if $\mu$ is defined by
\[
  \mu (F) = \begin{cases} \frac{C_1(F)}{C_0(F)} &\text{ if $C_0(F) \neq 0$} \\
  +\infty &\text{ if $C_0(F)=0$} \end{cases}
\]
where $C_0, C_1 : K(\Ac) \to \mathbb{Z}$ are a pair of group homomorphisms  satisfying: (i) $C_0(F) \geq 0 $ for any $F \in \Ac$; (ii) if $F \in \Ac$ satisfies $C_0(F)=0$, then $C_1 (F) \geq 0$.  The additive group $\mathbb{Z}$ in the definition of a slope-like function can be replaced by any  discrete additive subgroup of $\mathbb{R}$.  Whenever $\Ac$ is a noetherian abelian category, every slope-like function possesses the Harder-Narasimhan property \cite[Section 3.2]{LZ2}; we will then say an object $F \in \Ac$ is $\mu$-stable (resp.\ $\mu$-semistable) if, for every short exact sequence $0 \to M \to F \to N \to 0$ in $\Ac$ where $M,N \neq 0$, we have $\mu (M) < (\text{resp.} \leq \, ) \mu (N)$.

\section{Cohomological Fourier-Mukai transforms}\label{sec:cohomFMT}

In this section, we will assume conditions (1), (2) and (3a) in \ref{para:ellipfibdef}.  We will see how (3b) arises as a necessary condition for comparing slope stability and tilt stability under a Fourier-Mukai transform, and explain what is meant by taking a `limit' when we speak of  limit tilt stability.  The author thanks Ziyu Zhang for his  insight that the cohomological Fourier-Mukai transforms become easier to understand when the Chern characters are twisted before applying the transform.

\paragraph[The cohomological Fourier-Mukai transform formula for $\Phi$] Take any object $E \in D^b(X)$. By assumption \eqref{eq:assumptioncohomring}, we can write
\begin{align}
  \ch_{0}(E) &= n \notag\\
  \ch_{1}(E) &= x\Theta + p^\ast S \notag  \\
  \ch_{2}(E) &= \Theta p^\ast \eta + af \notag\\
  \ch_{3}(E) &= s \label{eq:chEtorsionfreepart}
\end{align}
for some $n,x \in \mathbb{Z}, a,s \in \mathbb{Q}$, $S, \eta \in A^1(B)_{\mathbb{Q}}$,  where $f$ denotes the class of a fiber of $p$.  Then
%Since applying the Fourier-Mukai transform $\Phi$ involves taking the intersection with the kernel of the transform, the torsion parts of $\ch(E)$ do not enter the computation of $\ch (\Phi E)$, and we have
\begin{align*}
  \ch_0 (\Phi E) &= x \\
  \ch_1 (\Phi E) &= -n\Theta + p^\ast \eta - \tfrac{1}{2} xc_1 \\
  \ch_2 (\Phi E) &= (\tfrac{1}{2} n c_1 - p^\ast S)\Theta + (s - \tfrac{1}{2} c_1 \Theta p^\ast \eta + \tfrac{1}{12} xc_1^2 \Theta )f \\
  \ch_3 (\Phi E) &= -\tfrac{1}{6} n\Theta c_1^2 -a +\tfrac{1}{2} \Theta c_1 p^\ast S
\end{align*}
where $c_1 = - p^\ast (K_B)$ from the formula \cite[(6.26)]{FMNT}.  Rewriting in terms of $K_B$, we obtain
\begin{align}
  \ch_0 (\Phi E) &= x \notag\\
  \ch_1 (\Phi E) &= -n\Theta + p^\ast (\eta + \tfrac{1}{2} xK_B) \notag\\
  \ch_2 (\Phi E) &= -\Theta p^\ast (S+\tfrac{1}{2}nK_B) + (s+\tfrac{1}{2}\eta K_B + \tfrac{1}{8}xK_B^2)f -\tfrac{1}{24}xK_B^2f \notag\\
  \ch_3 (\Phi E) &= -(a+\tfrac{1}{2}SK_B + \tfrac{1}{8}nK_B^2) - \tfrac{1}{24} nK_B^2. \label{eq:chPhiE}
\end{align}

\paragraph[A matrix notation for Chern characters] We introduce a notation for Chern characters that makes the cohomological Fourier-Mukai transform formula as easy to understand as in the product case of \cite{Lo14}.  The assumption \eqref{eq:assumptioncohomring} on the cohomology ring of $X$ says that it is the direct sum of the following six vector spaces:
\[
\begin{matrix}
  H^0(X,\mathbb{Q}) & p^\ast H^2(B,\mathbb{Q}) & p^\ast H^4(B,\mathbb{Q}) \\
  \Theta p^\ast H^0(B,\mathbb{Q}) & \Theta p^\ast H^2(B,\mathbb{Q}) & H^6(X,\mathbb{Q})
\end{matrix}.
\]
We can therefore think of the Chern character $\ch(E)$ in \eqref{eq:chEtorsionfreepart} as a matrix
\[
\begin{pmatrix}
n & p^\ast S & af \\
x\Theta & \Theta p^\ast \eta & s
\end{pmatrix}.
\]
Let us write $B$ to denote the $B$-field $B=\tfrac{1}{2}c_1 = -\tfrac{1}{2}p^\ast K_B$.    In the above matrix notation, twisting $\ch(E)$ by this $B$-field gives
\begin{equation}\label{eq:chEtwistedtfpart}
\ch^B(E) =
\begin{pmatrix}
n & p^\ast (S + \tfrac{1}{2}nK_B) & (a+\tfrac{1}{2}SK_B + \tfrac{1}{8} nK_B^2)f \\
x\Theta & \Theta p^\ast (\eta + \tfrac{1}{2}xK_B) & s+\tfrac{1}{2}\eta K_B + \tfrac{1}{8}xK_B^2
\end{pmatrix}
\end{equation}
and
\[
\ch(\Phi E) =
\begin{pmatrix}
x & p^\ast (\eta + \tfrac{1}{2}xK_B) & (s+\tfrac{1}{2}\eta K_B + \tfrac{1}{8}xK_B^2)f- \tfrac{1}{24}xK_B^2 f \\
-n\Theta & -\Theta p^\ast (S+\tfrac{1}{2}nK_B) & -(a+\tfrac{1}{2}SK_B + \tfrac{1}{8}nK_B^2) - \tfrac{1}{24}nK_B^2
\end{pmatrix}
\]
Conceptually, this makes the cohomological Fourier-Mukai transform very similar to the case of the product threefold (see \cite[4.1]{Lo14}):  In the product case, the cohomological FMT simply swaps the two rows of $\ch(E)$ in matrix notation, and changes the signs of the lower row (ignoring $\Theta$ and $f$ in the expressions).  The situation for Weierstra{\ss} threefolds here is similar: up to adding the  terms $- \tfrac{1}{24}xK_B^2 f \in A^2(X)$ and $- \tfrac{1}{24}nK_B^2 \in A^3(X)$, the cohomological FMT  also swaps the two rows of $\ch^B (E)$ and changes the signs of the lower row.  Indeed, if $K_B =\mathcal{O}_B$, then the above cohomological FMT formula reduces to that of the product case in \cite{Lo14}.

\paragraph[Polarisations on $X$] From \label{para:polarisationsonX} the Nakai-Moishezon Criterion \cite[Theorem 1.42]{KM}, the section $\Theta$ is a $p$-ample divisor on the Weierstra{\ss} threefold $X$.  Then, for any ample divisor $H_B$ on $B$, there exists  some $v_0>0$ such that $\Theta + vp^\ast H_B$ is an ample divisor on $X$ for all $v > v_0$ \cite[Proposition 1.45]{KM}.  We will always use polarisations of this form on $X$.

\paragraph[A numerical equivalence of cycles] On \label{para:Chowringequation} the elliptic threefold $X$, in order for us to compare slope stability  for a coherent sheaf $E$  with respect to a polarisation $\bo$, and  limit tilt stability for $\Phi E [1]$ with respect to another polarisation $\omega$,  we need to find a way to deform the polarisation $\omega$ so that, under this deformation of $\omega$ (which we can think of as taking a limit in $\omega$),  there exists a positive constant $C$ satisfying
\begin{equation}\label{eq:ch1EImoZwPhiE1asympequiv}
 C \cdot\Im \oZw (\Phi E [1]) \to  \bo^2 \ch_1^B(E)
\end{equation}
for any $E \in D^b(X)$.  By \ref{para:polarisationsonX}, we can write the polarisations $\bo, \omega$ on $X$ as
\begin{align*}
  \bo &= \bo_1 + \bo_2 \text{\quad with \quad } \bo_1 = a\Theta, \bo_2 = p^\ast \overline{H_B} \\
  \omega &= \omega_1 + \omega_2 \text{\quad with \quad} \omega_1 = u\Theta, \omega_2 = p^\ast \wh{H_B}
\end{align*}
where $a, u \in \mathbb{R}_{>0}$ and $\overline{H_B}, \wh{H_B}$ are ample classes on $B$.  With $\ch(E)$ as in \eqref{eq:chEtorsionfreepart}, we will also write
\begin{align*}
  A_1 &= x\Theta \\
  A_2 &= p^\ast (S+\tfrac{1}{2}nK_B) \\
  A_3 &= (s+\tfrac{1}{2}\eta K_B + \tfrac{1}{8} xK_B^2)f - \tfrac{1}{24}xK_B^2f
\end{align*}
to simplify some of the calculations to come.  Then
\begin{align}
  \bo^2 \ch_1^B(E) &= \bo^2 A_1 + \bo^2 A_2 \notag\\
  &= \bo^2 A_1 + \bo_1 (\bo_1 + 2\bo_2)A_2 \notag\\
  &= \bo^2 \Theta x + \bo_1 (\bo_1 + 2\bo_2)A_2 \label{eq:boch1Be-v1}
\end{align}
while
\begin{align}
  \Im \oZw (\Phi E [1]) &= \omega \ch_2 (\Phi E [1]) - \tfrac{\omega^3}{6} \ch_0 (\Phi E [1]) \notag\\
  &= \omega(\Theta A_2 - A_3) + \tfrac{\omega^3}{6}x \notag\\
  &= \omega\Theta A_2 -\omega_1 A_3 +  \tfrac{\omega^3}{6}x.  \label{eq:ImoZwPhiE1-v1}
\end{align}
Later in the article, we will let $u \to 0$, in which case we need the following numerical equivalence for the `coefficients' of $A_2$ and $x$ to hold in $A^2(X)$ in order to have a solution to  \eqref{eq:ch1EImoZwPhiE1asympequiv}:
\begin{equation}\label{eq:Chowringequation}
\frac{\bo_1 (\bo_1 + 2\bo_2)}{\Theta \bo^2} \equiv \frac{\omega\Theta}{\omega^3/6}.
\end{equation}

\paragraph[Solving the numerical equivalence]  By\label{para:solveChoweq} adjunction, we have
\begin{equation}\label{eq:Thetasquared}
  \Theta^2 = \Theta p^\ast K_B
\end{equation}
\cite[(6.6), Section 6.2.6]{FMNT}. In  order for there to be a solution to \eqref{eq:Chowringequation}, the 1-cycles
\[
  \bo_1 (\bo_1 + 2\bo_2) = a^2 \Theta p^\ast K_B + 2a\Theta p^\ast \overline{H_B} = a\Theta (ap^\ast K_B + 2p^\ast \overline{H_B})
\]
and
\[
\omega\Theta = (u\Theta + p^\ast \wh{H_B})\Theta = \Theta (up^\ast K_B + p^\ast \wh{H_B})
\]
must be real scalar multiples of each other under numerical equivalence.  Two ways through which this can happen are:
\begin{itemize}
\item[(a)] $K_B, \overline{H_B}$ are not $\mathbb{R}$-multiples of each other under numerical equivalence as cycles on $B$, as is the case for $K_B, \wh{H_B}$; however, $ap^\ast K_B + 2p^\ast \overline{H_B}$ and $up^\ast K_B + p^\ast \wh{H_B}$ are $\mathbb{R}$-multiples of each other up to numerical equivalence;
\item[(b)] $K_B, \overline{H_B}, \wh{H_B}$ are  $\mathbb{R}$-scalar multiples of one another up to numerical equivalence as cycles on $B$, such as when $B$ is a Fano, Enriques or K3 surface.
\end{itemize}

In case (a), we can solve \eqref{eq:Chowringequation} by  setting
\[
  a\wh{H_B} = 2u\overline{H_B}
\]
and then taking $\overline{H_B}, \wh{H_B}$ to be suitable $\mathbb{R}$-multiples of each other up to numerical equivalence.
Since we would like to let $u\to 0$ as mentioned in \ref{para:Chowringequation}, this solution would force $a \to 0$ as well.  This amounts to taking limits in both $\bo$ and $\omega$, and corresponds to the approach taken by Yoshioka \cite{YosAS, YosPII} on elliptic surfaces and the author's joint work with Zhang in \cite{LZ2}.  However, we would like to find a solution to \eqref{eq:Chowringequation} for an \textit{arbitrary} polarisation $\bo$, and so  will not pursue this approach in this article.

In case (b), let us write
\begin{equation}\label{eq:KbequivhHb}
  K_B \equiv hH_B
\end{equation}
for some ample class $H_B$ on $B$ and some $h \in \mathbb{R}$.  We allow the possibility of $h$ being zero, as is the case when $B$ is an Enriques or a K3 surface.  Then we write
\[
  \overline{H_B} = bH_B, \, \wh{H_B} = vH_B \text{\quad for some $b,v>0$}.
\]
Then
\begin{equation}\label{eq:boandomega}
  \bo = a\Theta + bp^\ast H_B, \text{\quad} \omega = u\Theta + vp^\ast H_B.
\end{equation}
Since the fibration $p : X \to B$ is a local complete intersection \cite[Section 6.2.1]{FMNT}, the numerical equivalence \eqref{eq:KbequivhHb} on $B$ pulls back to the numerical equivalence
\[
  p^\ast K_B \equiv hp^\ast H_B
\]
on $X$ \cite[Example 19.2.3]{Fulton}.  Using these notations, we now have

\begin{lemsub}\label{eq:curveforomegalimit}
If the parameters $a,b,u,v>0$ satisfy
\begin{equation}\label{eq:limitcurve}
\frac{a(ha+2b)}{(ha+b)^2} = \frac{(hu+v)}{\tfrac{1}{6}u(h^2u^2 + 3huv + 3v^2)}.
\end{equation}
then the numerical equivalence \eqref{eq:Chowringequation} holds.
\end{lemsub}

\begin{proof}
Observe that
\begin{align*}
  \bo_1 (\bo_1 + 2\bo_2) &= a\Theta (ap^\ast K_B + 2p^\ast \overline{H_B}) \\
  &\equiv a(ha+2b)\Theta p^\ast H_B, \\
  \Theta \bo^2 &= \Theta (\bo_1 + \bo_2)^2 \\
  &= (h^2a^2 + 2hab + b^2)\Theta p^\ast H_B^2 \\
  &= (ha+b)^2 \Theta p^\ast H_B^2, \\
  \omega \Theta &\equiv hu \Theta p^\ast H_B + v\Theta p^\ast H_B \\
  &= (hu+v)\Theta p^\ast H_B, \\
  \tfrac{\omega^3}{6} &= \tfrac{1}{6} (\omega_1+\omega_2)^3 \\
  &= \tfrac{1}{6}\omega_1 (\omega_1^2 + 3\omega_1\omega_2 + 3\omega_2^2) \\
  &= \tfrac{1}{6}u(h^2u^2\Theta p^\ast H_B^2 + 3huv\Theta p^\ast H_B^2 + 3v^2 \Theta p^\ast H_B^2) \\
  &= \tfrac{1}{6}u(h^2u^2 + 3huv + 3v^2)\Theta p^\ast H_B^2.
\end{align*}
The numerical equivalence \eqref{eq:Chowringequation}  can now be rewritten as
\[
\frac{a(ha+2b)\Theta p^\ast H_B}{(ha+b)^2 \Theta p^\ast H_B^2} \equiv \frac{(hu+v)\Theta p^\ast H_B}{\tfrac{1}{6}u(h^2u^2 + 3huv + 3v^2)\Theta p^\ast H_B^2},
\]
which clearly holds if \eqref{eq:limitcurve} is satisfied.
\end{proof}

 For fixed $a,b>0$ and fixed $h \in \mathbb{R}$, if we ensure
 \begin{equation}\label{eq:constraint1}
 ha+2b>0 \text{\quad and\quad } ha+b \neq 0,
  \end{equation}
  and $u>0$ is bounded from above, then for any $v \gg 0$, we can always find some $u$ that satisfies \eqref{eq:limitcurve}.  Furthermore, as $v \to \infty$ we have $u=O(\tfrac{1}{v})$ subject to the constraint \eqref{eq:limitcurve}.  %From now on, we will set
%\[
%  \alpha = \tfrac{b}{a}
%\]
%so that
%\[
%  \bo = \tfrac{b}{\alpha} \Theta + b p^\ast H_B.
%\]
%Then

\paragraph[Relation to the product threefold case]
When $h=0$, i.e.\ when the surface $B$ is  numerically $K$-trivial, the equation \eqref{eq:limitcurve} reduces to
\[
 \tfrac{b}{a} = uv,
\]
which is exactly the relation used in defining limit tilt stability on a product elliptic threefold $C \times B$ in \cite[2.4]{Lo14}, where $C$ is a smooth elliptic curve, $B$ a K3 surface, and $X$ is considered as a trivial fibration over $B$ via the second projection.

\section{Limit tilt stability}\label{sec:limittiltstabconstruction}

We define limit tilt stability in this section.  From now on through the end of the article, we will assume conditions (1), (2) and (3) in \ref{para:ellipfibdef}; in particular, we will fix $H_B$ and $h$ as in condition (3).

Recall from \eqref{eq:boandomega} that $u,v >0$ are two parameters in our notation for the polarisation
\[
\omega = u\Theta + vp^\ast H_B
\]
on the threefold $X$, where $H_B$ is an ample class on the base $B$.  In this section, we define a heart of t-structure that can be regarded as a `limit' of the heart $\Bc_\omega$ as $v \to \infty$ along the curve \eqref{eq:limitcurve} on the $vu$-plane.

\begin{lem}\label{lem:muwleftcomponent}
There exists an $m>0$ such that, for any $k\geq m$ and any effective divisor $D$ on $X$, we have $(\Theta p^\ast H_B + kf)D \geq 0$.
\end{lem}

\begin{proof}
Choose any $m_0>0$ such that
\begin{itemize}
\item $m_0 > v'$ where $v'$ is as in assumption (3a) in \ref{para:ellipfibdef}, so that  $\Theta + kp^\ast H_B$ is ample on $X$ for any $k \geq m_0$.
\item  $h+2m_0>0$
\end{itemize}
  Then for any effective divisor $D$ on $X$ and any $k \geq m_0$,  we have $D(\Theta + kp^\ast H_B)^2 \geq 0$.  Moreover,
\begin{align*}
  (\Theta + kp^\ast H_B)^2 &= \Theta^2 + 2k\Theta p^\ast H_B + k^2 p^\ast H_B^2 \\
  &= (h+2k)\Theta p^\ast H_B + k^2 p^\ast H_B^2 \\
  &= (h+2k)(\Theta p^\ast H_B + \tfrac{k^2}{h+2k}p^\ast H_B^2).
\end{align*}
Note that $\tfrac{d}{dk} \, \tfrac{k^2}{h+2k} = \tfrac{2k(h+k)}{(h+2k)^2}$, so under our assumption on $k$, the derivative $\tfrac{d}{dk} \, \tfrac{k^2}{h+2k}$ is always positive.  Thus we can take $m$ in the lemma to be $\tfrac{m_0^2}{h+2m_0} H_B^2$.
\end{proof}

\begin{cor}\label{cor:petaeffective}
If $\eta \in A^1(B)$ and $p^\ast \eta$ is an effective divisor class in $A^1(X)$, then $H_B \eta \geq 0$, with equality if and only if $\eta =0$.
\end{cor}

\begin{proof}
Suppose $\eta \in A^1(B)$ and $p^\ast \eta$ is an effective divisor class on $X$.  Fix any $k>0$ such that $h+2k>0$ and $\Theta + kp^\ast H_B$ is an ample class on $X$.  Then from the proof of Lemma \ref{lem:muwleftcomponent}, we know $(\Theta + kp^\ast H_B)^2 p^\ast \eta$ is a positive multiple of $(\Theta p^\ast H_B + \tfrac{k^2}{h+2k}p^\ast H_B^2)p^\ast \eta = \Theta (p^\ast H_B)(p^\ast \eta)=H_B\eta$.  Hence $H_B \eta \geq 0$, with equality iff $p^\ast \eta=0$.  Since $p^\ast \eta = 0$ implies $0=p_\ast (\Theta (p^\ast \eta))= \eta$ by the projection formula, the corollary follows.
\end{proof}

As a consequence of Lemma \ref{lem:muwleftcomponent}, we can find some $m>0$ such that the following slope-like function  on $\Coh (X)$ has the Harder-Narasimhan property:
\begin{equation}
  \mu_{\Theta p^\ast H_B + mf}(E) = \frac{(\Theta p^\ast H_B + mf)\ch_1(E)}{\ch_0(E)} \text{\quad for $E \in \Coh (X)$}.
\end{equation}
(From the proof of Lemma \ref{lem:muwleftcomponent}, this function is merely a multiple of the usual slope function with respect to some ample class on $X$.)

\begin{lem}\label{lem:4-1n4-3paper14analogue}
Take any $u_0>0$ and $F \in \Coh (X)$.  Suppose $m>0$ is as in Lemma \ref{lem:muwleftcomponent} and $\omega$ is of the form \eqref{eq:boandomega}.
\begin{itemize}
\item[(1)] The following are equivalent:
  \begin{itemize}
  \item[(a)] There exists $v_0>0$ such that $F \in \Fc_\omega$ for all $(v,u) \in (v_0,\infty) \times (0,u_0)$.
  \item[(b)] There exists $v_0>0$ such that, for every nonzero subsheaf $A \subseteq F$, we have $\mu_\omega (A) \leq 0$ for all $(v,u) \in (v_0,\infty)\times (0,u_0)$.
  \item[(c)] For every nonzero subsheaf $A \subseteq F$, either (i) $\mu_f (A) < 0$, or (ii) $\mu_f (A)=0$ and \\ $\mu_{\Theta p^\ast H_B + mf}(A) \leq 0$.
  \end{itemize}
\item[(2)] The following are equivalent:
  \begin{itemize}
  \item[(a)] There exists $v_0>0$ such that $F \in \Tc_\omega$ for all $(v,u) \in (v_0,\infty) \times (0,u_0)$.
  \item[(b)]  There exists $v_0>0$ such that, for every nonzero sheaf quotient $F \twoheadrightarrow A$, we have $\mu_\omega (A) > 0$ for all $(v,u) \in (v_0,\infty) \times (0,u_0)$.
  \item[(c)] For any nonzero sheaf quotient $F \twoheadrightarrow A$, either (i) $\mu_f (A)>0$, or (ii) $\mu_f(A)=0$ and $\mu_{\Theta p^\ast H_B + mf}(A) >0$.
  \end{itemize}
\end{itemize}
\end{lem}

\begin{proof}
The proofs for parts (1) and (2) are essentially the same as those for \cite[Lemmas 4.1, 4.3]{Lo14}, so we will only point out the modifications needed, using  part (1) as an example.

For part (1), the directions (a) $\Leftrightarrow$ (b) are clear.  Observe that
\begin{align}
\omega^2 &= (u\Theta + vp^\ast H_B)^2 \notag\\
&= u^2\Theta^2 + 2uv\Theta p^\ast H_B + v^2 p^\ast H_B^2 \notag\\
&\equiv hu^2 \Theta p^\ast H_B + 2uv \Theta p^\ast H_B + v^2 p^\ast H_B^2 \notag\\
&= u(hu+2v)\Theta p^\ast H_B + v^2 p^\ast H_B^2 \label{eq:omegasquared}\\
&= u(hu+2v)(\Theta p^\ast H_B +mf) + (v^2H_B^2 -u(hu+2v)m)f. \notag
\end{align}
Hence  for any $A \in \Coh (X)$
\begin{align}
 \mu_\omega (A) &= \tfrac{\omega^2\ch_1(A)}{\ch_0(A)} \notag\\
 &= u(hu+2v)\mu_{\Theta p^\ast H_B+mf}(A) + (v^2H_B^2-u(hu+2v)m)\mu_f(A) \label{eq:muomegadecomposition}\\
 &= O(v)\mu_{\Theta p^\ast H_B+mf}(A) + O(v^2)\mu_f(A) \text{\quad as $v \to\infty$ while $u$ is fixed}.\notag
\end{align}
Since the coefficients of $\mu_{\Theta p^\ast H_B + mf}(A)$ and $\mu_f(A)$ above are positive for $v \gg 0$, the implication (b) $\Rightarrow$ (c)  follows easily.  To see  (c) $\Rightarrow$ (b), use the same argument as in \cite[Lemma 4.1]{Lo14} and note that both $\mu_f, \mu_{\Theta p^\ast H_B + mf}$ have the HN property on $\Coh (X)$.
\end{proof}

\paragraph Since the conditions in parts (1)(c) and (2)(c) of Lemma \ref{lem:4-1n4-3paper14analogue} are independent of the polarisation $\omega = u\Theta + vp^\ast H_B$, we can define the following categories:
\begin{itemize}
  \item $\Tc^l$, the extension closure of all coherent sheaves satisfying  (2)(c) in Lemma \ref{lem:4-1n4-3paper14analogue},
  \item $\Fc^l$, the extension closure of all coherent sheaves satisfying  (1)(c) in Lemma \ref{lem:4-1n4-3paper14analogue},
\end{itemize}
and subsequently  the extension closure
\begin{equation*}
  \Bl = \langle \Fc^l [1], \Tc^l \rangle.
\end{equation*}
As in the case of the product elliptic threefold \cite[Remark 4.4(vi)]{Lo14}, it is easy to check that the categories $\Tc^l, \Fc^l$ can equivalently be defined as follows, for any fixed $a,b>$ satisfying \eqref{eq:constraint1} and for $v,u>0$:
\begin{itemize}
  \item $\Tc^l$: the extension closure of all coherent sheaves $F$ on $X$ such that $F \in \Tc_\omega$ for all $v \gg 0$ subject to \eqref{eq:limitcurve},
  \item $\Fc^l$: the extension closure of all coherent sheaves $F$ on $X$ such that $F \in \Fc_\omega$ for all $v \gg 0$ subject to \eqref{eq:limitcurve}
\end{itemize}

\paragraph The following \label{para:TlFlbasicproperties} properties  of $\Tc^l, \Fc^l$ and $\Bl$ are immediate.  We will omit their proofs since they are analogous  to their counterparts in the product case (see \cite[Remark 4.4]{Lo14}):
\begin{itemize}
\item[(i)] All the torsion sheaves on $X$ are contained in $\Tc^l$.
\item[(ii)] $\Fc^l$ is contained in the category of torsion-free sheaves on $X$.
\item[(iii)] $W_{0,\wh{\Phi}} \subset \Tc^l$.
\item[(iv)] $f\ch_1 (E) \geq 0$ for every $E \in \Bl$.
\item[(v)] $\Fc^l \subset W_{1,\wh{\Phi}}$; this follows from Lemma \ref{lem:paper14lem4-6analogue} below.
\end{itemize}

\begin{lem}\label{lem:paper14lem4-6analogue}
$(\Tc^l, \Fc^l)$ is a torsion pair in $\Coh (X)$, and the category $\Bl$ is the heart of a t-structure on $D^b(X)$.
\end{lem}

\begin{proof}
For any $A \in \Coh (X)$, the decomposition of $\mu_\omega$ in \eqref{eq:muomegadecomposition} gives
\[
 \mu_\omega (A) = O(1) \mu_{\Theta p^\ast H_B+mf}(A) + O(v^2)\mu_f(A) \text{\quad as $v \to \infty$ along \eqref{eq:limitcurve}}.
\]
The proof of \cite[Lemma 4.6]{Lo14} carries over if we replace $s, t$ by $v, u$, replace the curve $ts = \alpha$ by \eqref{eq:limitcurve}, and replace the function $\mu^\ast$ by $\mu_{\Theta p^\ast H_B+mf}$ in that proof.
\end{proof}

\begin{lem}\label{lem:paper14lem4-7analogue}
For any  $a,b > 0$ satisfying \eqref{eq:constraint1}, and for any $E \in \Bl$, we have
\[
  \oZw (E) \in -i\mathbb{H}_0 \text{\quad for $v \gg 0$}
\]
whenever $u,v>0$ satisfy the constraint \eqref{eq:limitcurve}.
\end{lem}

\begin{proof}
Take any $E \in \Bl$.  From Lemma \ref{lem:4-1n4-3paper14analogue}, we know there exists some $v_0>0$ such that, for any $u>0, v>v_0$ satisfying \eqref{eq:limitcurve}, we have $E \in \Bc_\omega$, in which case $\oZw (E) \in -i\mathbb{H}_0$.
\end{proof}

\paragraph[$\oZl$-stability on $\Bl$ as asymptotically polynomial stability] Lemma \ref{lem:paper14lem4-7analogue} allows \label{para:oZlstabasymppolystab} us to define the notion of $\oZl$-stability on $\Bl$ for any fixed $a, b$ satisfying \eqref{eq:constraint1}.  Take any $F \in \Bl$, and suppose $u,v>0$ satisfy \eqref{eq:limitcurve}.  Supposing $\ch(F)$ is of the form
\begin{equation}\label{eq:chF}
  \ch (F) = \begin{pmatrix} \wt{n} & p^\ast \wt{S} & \wt{a}f \\
  \wt{x}\Theta & \Theta p^\ast \wt{\eta} & \wt{s}
  \end{pmatrix},
\end{equation}
we have
\begin{align*}
  \tfrac{\omega^2}{2} \ch_1(F) &= \tfrac{1}{2}(u(hu+2v)\Theta p^\ast H_B + v^2 p^\ast H_B^2)(\wt{x}\Theta + p^\ast \wt{S}) \text{\quad by \eqref{eq:omegasquared}}\\
  &= \tfrac{1}{2}(hu(hu+2v)+v^2)\wt{x}\Theta p^\ast H_B^2 + \tfrac{1}{2}u(hu+2v)\Theta p^\ast (H_B \wt{S}), \\
  \omega \ch_2(F) &= (u\Theta + vp^\ast H_B)(\Theta p^\ast \wt{\eta} + \wt{a}f) \\
  &= (hu+v)\Theta p^\ast (H_B\wt{\eta}) + u\wt{a}
\end{align*}
giving us
\begin{align}
  \oZw (F) &= \tfrac{\omega^2}{2}\ch_1(F) + i(\omega\ch_2(F) - \tfrac{\omega^3}{6}\ch_0(F))\notag \\
  &= \left( \tfrac{hu(hu+2v)+v^2}{2}H_B^2\wt{x}+\tfrac{u(hu+2v)}{2}H_B\wt{S}\right) \label{eq:oZwFcomplete}\\
  &\hspace{2cm} + i \left( (hu+v)H_B\wt{\eta} + u\wt{a} - \tfrac{u(h^2u^2+3huv+3v^2)}{6}H_B^2\wt{n}\right).
\end{align}
As $v \to \infty$ under the constraint \eqref{eq:limitcurve}, we have
\begin{equation}\label{eq:oZwFOvform}
  \oZw (F) = O(v^2)H_B^2\wt{x} + O(1)H_B\wt{S} + i \left( O(v)H_B\wt{\eta} + O(\tfrac{1}{v})\wt{a}-O(v)H_B^2\wt{n} \right).
\end{equation}
Therefore, for $v \gg 0$ we can define a function $\phi (F)(v)$ by the relation
\begin{equation}\label{eq:Fphasedefinition}
  \oZw (F) \in \mathbb{R}_{>0}e^{i\pi\phi (F)(v)} \text{\quad  if $\oZw (F) \neq 0$},
\end{equation}
while we set  $\phi (F)(v)=\tfrac{1}{2}$ in case of $\oZw (F) =0$.  Here, $\phi (F)(v)$ is a function germ
\[
  \phi (F) : \mathbb{R} \to (0,1]
\]
as in the definition of Bayer's polynomial stability  \cite[Section 2.1]{BayerPBSC}.  Even though $u$ is only an implicit function in $v$ under the relation \eqref{eq:limitcurve}, for $v \gg 0$ we can write $u$ as a function in $v$ by requiring $u>0$, in which case we have $u = O(\tfrac{1}{v})$ as mentioned in \ref{para:solveChoweq}.

In fact, we can write $u$ as a  power series in $\tfrac{1}{v}$ that is convergent on $v>v_0$, for some fixed $v_0 \in \mathbb{R}_{>0}$.  To see this, note that $u$ is considered as an implicit function in $v$ via \eqref{eq:limitcurve}, which can be rewritten as
\begin{equation}\label{eq2}
  \tfrac{A}{6h}u(h^2u^2w^2 + 2huw+3)-\tfrac{1}{h}w - w^2u=0
\end{equation}
if we set $A = \frac{a(ha+2b)}{(ha+b)^2}$ and make the change of variable $w=\tfrac{1}{v}$.  Writing $F(u,w)$ to denote the left-hand side of \eqref{eq2}, we see that $F(0,0)=0$ and $F(u,w)$ contains the term $\tfrac{3A}{6h}u$ which is a nonzero multiple of a positive power of $u$.  By \cite[Theorem 1.1, Remark 1.2(2)]{KollarLRS}, $u$ is thus a convergent Puiseux series in $w$, which in our case is a power series.  That is, $u$ is a convergent power series in $\tfrac{1}{v}$.

As a result, for any nonzero objects $E, F$ in $\Bl$, the polynomial phase functions $\phi (E), \phi (F)$ are comparable with respect to $\prec$ \cite[Lemma 4.3]{Lo20}: Given two objects $E, F \in \Bl$, we will write
$\phi (E) \prec \phi (F)$ (resp.\ $\phi (E) \preceq \phi (F)$) if $\phi (E)(v) < \phi (F)(v)$ (resp.\ $\phi (E)(v) \leq \phi (F)(v)$) as $v \to \infty$ along \eqref{eq:limitcurve}.  This allows  us  to define the following notion of stability, which we can think of as an `asymptotically polynomial stability'.

\begin{defn}
 We say an object $F \in \Bl$ is  $\oZl$-stable or limit tilt stable (resp.\ $\oZl$-semistable or limit tilt semistable) if, for  every $\Bl$-short exact sequence of the form
\[
0 \to M \to F \to N \to 0
\]
where $M,N\neq 0$, we have
\[
  \phi (M) \prec \phi (N) \text{\quad (resp.\ $\phi (M) \preceq \phi (N)$)}.
\]
\end{defn}

Even though  tilt stability and limit tilt stability both use the formula of   $\oZl$ for the central charge, there is a fundamental difference between them.  Tilt stability, initially defined in \ref{para:slopetiltstab}  and reformulated in \ref{para:tiltstab2}, uses a fixed ample class $\omega$ in the central charge $\oZl$.  Therefore, in the case of tilt stability, we consider the central charge  $\oZl$  as a group homomorphism $K(X) \to \mathbb{C}$.

In contrast, even though limit tilt stability also uses $\oZl$ as the central charge, we think of $\omega$ as being of the form $\omega = u\Theta + vp^\ast H_B$ where $u, v>0$ are parameters that must satisfy the relation \eqref{eq:limitcurve}.   Since $u$ can be solved as  a power series in $\tfrac{1}{v}$, we think of $\oZl$ as a group homomorphism $K(X) \to \mathbb{C}\llbracket \tfrac{1}{v}\rrbracket$ that takes values in the power series ring in $\tfrac{1}{v}$.  For limit tilt stability, we compare phases of objects for $v \gg 0$ (i.e.\ when $\omega$ approaches the fiber direction) rather than for a fixed value of $\omega$.  Lemma \ref{lem:limvsnonlimtiltstab} describes in a  precise manner a relation  between tilt stability and limit tilt stability.

\paragraph We set \label{para:TFlastdefinitions}
\begin{align*}
\Tc^{l,+} &= \langle F \in \Tc^l : F \text{ is $\mu_f$-semistable}, \infty > \mu_f (F) > 0\rangle\\
\Tc^{l,0} &= \{  F \in \Tc^l : F \text{ is $\mu_f$-semistable}, \mu_f(F)=0\}\\
\Fc^{l,0} &= \{ F \in \Fc^l : F \text{ is $\mu_f$-semistable}, \mu_f(F)=0\} \\
\Fc^{l,-} &= \langle F \in \Fc^l : F \text{ is $\mu_f$-semistable}, \mu_f(F)<0\rangle.
\end{align*}
These categories have the same basic properties as their counterparts in the product threefold case in \cite{Lo14}, and so we omit their proofs and only indicate where the counterparts appear in \cite{Lo14}:
\begin{itemize}
\item[(i)] $\Tc^{l,0} \subset W_{1,\wh{\Phi}}$  \cite[Remark 4.8(iii)]{Lo14}.
\item[(ii)]   $\wh{\Phi}\Bl \subset D^{[0,1]}_{\Coh (X)}$, which  gives us a torsion triple in $\Bl$
\begin{equation}\label{eq:Bltorsiontriple1}
( \Fc^l[1], \, W_{0,\wh{\Phi}}, \, W_{1,\whPhi} \cap \Tc^{l,0} )
\end{equation}
 \cite[Remark 4.9]{Lo14}.
\item[(iii)] We have a torsion quintuple in $\Bl$
\begin{equation*}
  ( \Fc^{l,0}[1],\, \Fc^{l,-}[1],\, \Coh^{\leq 2}(X),\, \Tc^{l,+},\, \Tc^{l,0} )
\end{equation*}
 \cite[Lemma 4.11]{Lo14}.
\item[(iv)] $W_{1,\wh{\Phi}} \cap \Tc^l = \langle W_{1,\wh{\Phi}} \cap \Coh^{\leq 2}(X), \Tc^{l,0} \rangle$, and every object in this category can be written as the extension of a sheaf in $\Tc^{l,0}$ by a sheaf in $W_{1,\wh{\Phi}} \cap \Coh^{\leq 2}(X)$.  Moreover, for every $F \in W_{1,\wh{\Phi}} \cap \Tc^l$, the transform $\wh{\Phi} F [1]$ is a torsion sheaf on $X$.  (The proof of these claims is the same as the proof of \cite[Lemma 4.10]{Lo14} - we can just replace $\ch_{10}$ in that proof by $f\ch_1$.)
\item[(v)] For every torsion-free sheaf $E$ on $X$, we have $\Phi E [1] \in \Bl$   \cite[Lemma 4.12]{Lo14}.
\end{itemize}

In the predecessor to this article, \cite{Lo14}, limit tilt stability (or $\oZl$-stability) was referred to as $\nu^l$-stability.

\section{Phases of  objects with respect to limit tilt stability}\label{sec:phasecomputations}

We  compute the phases of various objects in $\Bl$ in this section.  These computations will be needed in order to establish the Harder-Narasimhan property of $\oZl$-stability on $\Bl$.  Throughout this section, we will assume that  $\bo, \omega$ are ample classes of the form \eqref{eq:boandomega} with $a,b,u,v >0$ satisfying \eqref{eq:constraint1} and \eqref{eq:limitcurve} and where $a,b$ are fixed.

\paragraph Take \label{eq:ImoZwPhiE1bo2ch1BEcomparison} any object $E \in D^b(X)$ with $x = f\ch_1(E) \neq 0$.  Using the notation in \ref{para:Chowringequation}, we have
\begin{align*}
  \Im \oZw (\Phi E [1]) &= \tfrac{\omega^3}{6}x + \omega \Theta A_2    -  \omega_1 A_3 \text{\quad from \eqref{eq:ImoZwPhiE1-v1}}\\
  &= \tfrac{\omega^3}{6}x \left( 1 + \tfrac{\omega \Theta A_2}{(\omega^3/6)x} \right) - \omega_1 A_3 \\
  &= \tfrac{\omega^3}{6}x \left( 1 + \tfrac{\bo_1 ( \bo_1 + 2\bo_2)A_2}{\Theta \bo^2 x} \right) - \omega_1 A_3 \text{\quad by \eqref{eq:Chowringequation}} \\
  &= \tfrac{(\omega^3/6)x}{\Theta \bo^2 x} (\Theta \bo^2 x + \bo_1 (\bo_1 + 2\bo_2)A_2 ) - \omega_1 A_3 \\
  &= \tfrac{(\omega^3/6)}{\Theta \bo^2 }\, \bo^2 \ch_1^B(E) - \omega_1 A_3 \text{\quad by \eqref{eq:boch1Be-v1} and where $B=-\tfrac{1}{2}p^\ast K_B$} \\
  &= O(v)\bo^2\ch_1^B(E) - O(\tfrac{1}{v})\Theta A_3 \text{\quad as $v \to \infty$}.
\end{align*}

\paragraph Suppose \label{para:WIT1shfch1eqzero} $F$ is a $\wh{\Phi}$-WIT$_1$ sheaf on $X$ with $\ch_0(F)\neq 0$ and $f\ch_1(F)=0$.  Then $\wh{F} = \wh{\Phi} F[1]$ is a coherent sheaf supported in dimension 2, and we have $0 < \bo \ch_1(\wh{F})$.  From the computation in \ref{eq:ImoZwPhiE1bo2ch1BEcomparison}, for $v \gg 0$ we have $0<\Im \oZw (\Phi \wh{F} [1]) = \Im \oZw (\Phi \wh{\Phi} F [2] ) = \Im \oZw (F[1])$ and $\Im \oZw (F)$ is $O(v)$.

\paragraph[The cohomological Fourier-Mukai formula for $\wh{\Phi}$]  Suppose $F \in D^b(X)$ has Chern character of the form \eqref{eq:chF}, i.e.\
\begin{equation*}
  \ch (F) = \begin{pmatrix} \wt{n} & p^\ast \wt{S} & \wt{a}f \\
  \wt{x}\Theta & \Theta p^\ast \wt{\eta} & \wt{s}
  \end{pmatrix}.
\end{equation*}
Then the formula \cite[(6.27)]{FMNT} gives
\begin{align}
  \ch_0 (\wh{\Phi}F) &= \wt{x} \notag\\
  \ch_1 (\wh{\Phi}F) &= -\wt{n}\Theta + p^\ast \wt{\eta} - \tfrac{1}{2}\wt{x}p^\ast K_B \notag\\
  \ch_2 (\wh{\Phi}F) &= (\tfrac{1}{2}\wt{n}p^\ast K_B - p^\ast \wt{S})\Theta + (\wt{s} - \tfrac{1}{2}K_B\wt{\eta} + \tfrac{1}{12}\wt{x}K_B^2)f \notag\\
  \ch_3 (\wh{\Phi}F) &= -\tfrac{1}{6}\wt{n}K_B^2 - \wt{a} +\tfrac{1}{2}K_B\wt{S}. \label{eq:chwhPhiF}
\end{align}

The following is a generalisation of \cite[Remark 5.17]{LZ2} from the product elliptic threefold to a general Weierstra{\ss} elliptic threefold:
\begin{lem}\label{lem:LZ2Remark5-17analogue}
\begin{itemize}
\item[(a)] For $1 \leq d \leq 3$, if $F \in \Coh^d(p)_{d-1} \cap W_{1,\wh{\Phi}}$  then $\ch_{3-d+1}(F)p^\ast H_B^{d-1} \leq 0$.
\item[(b)] For $1 \leq d \leq 3$, if $F \in \Coh^d(p)_{d-1} \cap W_{0,\wh{\Phi}}$  then $\ch_{3-d+1}(F)p^\ast H_B^{d-1} > 0$.
\end{itemize}
\end{lem}

\begin{proof}
Suppose $F \neq 0$ and $\ch(F)$ is as in \eqref{eq:chF}. Then
\begin{equation*}
  \ch_{3-d+1}(F) p^\ast H_B^{d-1} = \begin{cases}
  \ch_1(F)p^\ast H_B^2 =  \wt{x} H_B^2 &\text{ when $d=3$} \\
  \ch_2(F)p^\ast H_B = \wt{\eta} H_B &\text{ when $d=2$} \\
  \ch_3(F) = \wt{s} &\text{ when $d=1$}
  \end{cases}.
\end{equation*}

We consider parts (a) and (b) at the same time for each $d$:
\begin{itemize}
\item When $d=3$: $\ch(\wh{\Phi}F) = {\left( \begin{smallmatrix} \wt{x} & \ast & \ast \\
-\wt{n}\Theta & \ast & \ast \end{smallmatrix}\right)}$, so $\ch_1(F)p^\ast H_B^2 = \wt{x}H_B^2=\ch_0 (\wh{\Phi} F)H_B^2$, which is  nonpositive (resp.\ nonnegative) if $F$ is $\wh{\Phi}$-WIT$_1$ (resp.\  $\wh{\Phi}$-WIT$_0$).  Furthermore, if $F$ is $\wh{\Phi}$-WIT$_0$, then we cannot have $\ch_0(\wh{\Phi}F)=0$, or else $\wh{\Phi}F$ would be a $\Phi$-WIT$_1$ sheaf in $\Coh^{\leq 2}(X)$, implying $\wh{\Phi} F$ and hence $F$ itself lies in $\Coh (p)$ \cite[Lemma 2.6]{Lo7}, contradicting the assumption that $F$ is supported in dimension 3.  This proves the case $d=3$.
\item When $d=2$: $\wt{n}=\wt{x}=0$ by assumption and $\ch(\wh{\Phi}F) = {\left( \begin{smallmatrix} 0 & p^\ast \wt{\eta} & (\wt{s} - \tfrac{1}{2}K_B\wt{\eta} )f \\
    0 & -\Theta p^\ast \wt{S} & \ast \end{smallmatrix}\right)}$.  Then $\ch_2(F)p^\ast H_B = H_B\wt{\eta} = (p^\ast \wt{\eta})\Theta p^\ast H_B$, which has the same sign as the leading coefficient of $\omega^2 \ch_1 (\wh{\Phi}F)$ as $v \to \infty$.  Therefore, $\ch_2(F)p^\ast H_B$ is nonpositive (resp.\ nonnegative) if $F$ is $\wh{\Phi}$-WIT$_1$ (resp.\ $\wh{\Phi}$-WIT$_0$).  Furthermore, if $F$ is $\wh{\Phi}$-WIT$_0$, then we cannot have $\ch_2(F)p^\ast H_B=0$, or else $\ch_1(\wh{\Phi}F)=0$, implying $\wh{\Phi}F$ is a $\Phi$-WIT$_1$ coherent sheaf supported in dimension 1, in which case $\wh{\Phi} F$ and hence $F$ itself must be a fiber sheaf (by \cite[Remark 3.14, Lemma 3.15]{Lo7}), contradicting the assumption that $F$ is supported in dimension 2.
\item When $d=1$: $\wt{n}, \wt{x}, p^\ast \wt{S}, \Theta p^\ast \wt{\eta}$ all vanish, and $\ch(F) = {\left( \begin{smallmatrix} 0 & 0 & \wt{a}f \\ 0 & 0 & \wt{s}\end{smallmatrix}\right)}$.  When $F$ is $\wh{\Phi}$-WIT$_i$, we have that $\ch_2(\wh{\Phi}F[i])=\wt{s}f$ is effective, and so $(-1)^i \wt{s} \geq 0$.  Furthermore, if $F$ is $\wh{\Phi}$-WIT$_0$ then we must have $\wt{s}>0$, or else $\wh{\Phi} F$ is a $\Phi$-WIT$_1$ fiber sheaf with $\ch_2=0$, meaning $\whPhi F$ is $\Phi$-WIT$_1$ and yet is supported in dimension 0, which is impossible by \cite[Proposition 6.38]{FMNT}.
\end{itemize}
\end{proof}

\paragraph Adopting \label{para:boxdiagramsdefinition} the notation in \cite[Section 3]{Lo14}, we will define the following subcategories of $\Coh (X)$:
\begin{align*}
   \scalea{\gyoung(;;;,;;;+)}&= \Coh^{\leq 0}(X) \\
    \scalea{\gyoung(;;;+,;;;+)} &= \{ E \in \Coh^1(p)_0 :  \text{ all $\mu$-HN factors of $E$ have $\infty>\mu>0$}\} \\
  \scalea{\gyoung(;;;+,;;;0)} &=\{ E \in \Coh^1(p)_0 : \text{ all $\mu$-HN factors of $E$ have $\mu=0$}\} \\
  \scalea{\gyoung(;;;+,;;;-)} &= \{ E \in \Coh^1(p)_0 :  \text{ all $\mu$-HN factors of $E$ have $\mu<0$}\} \\
   \scalea{\gyoung(;;;*,;;+;*)} &= \Coh^1(p)_1 \cap \{\Coh^{\leq 0}\}^\uparrow\\
    \scalea{\gyoung(;;+;*,;;+;*)} &= \Coh^2(p)_1 \cap W_{0,\whPhi}. \\
\scalea{\gyoung(;;+;*,;;0;*)} &= \{ E \in \Phi (\{\Coh^{\leq 0}\}^\uparrow \cap \Coh^{\leq 1}(X)) : \dimension E = 2\} \\
\scalea{\gyoung(;;+;*,;;-;*)} &= \{ E \in \Coh^2(p)_1 \cap W_{1,\wh{\Phi}} : \Theta (p^\ast {\eta})\neq 0\} \text{ where $\eta$ is as in \eqref{eq:chEtorsionfreepart}}\\
\scalea{\gyoung(;;*;*,;+;*;*)}  &= \{ E \in \Coh^2(p)_2 \cap W_{0,\wh{\Phi}} : f\ch_1(E)>0\}
\end{align*}

The Fourier-Mukai transform $\Phi$ induces the following equivalences, as already observed in \cite[Remark 3.1]{Lo14}:
\[
  \xymatrix @-0.7pc{
\scalea{\gyoung(;;;,;;;+)} \ar[dr] & \scalea{\gyoung(;;;+,;;;+)} \ar@/^1.5pc/[dd] & \scalea{\gyoung(;;;*,;;+;*)} \ar[dr] & \scalea{\gyoung(;;+;*,;;+;*)} \ar@/^1.5pc/[dd] & \scalea{\gyoung(;;*;*,;+;*;*)} \ar[dr] & \scalea{\gyoung(;+;*;*,;+;*;*)} \ar@/^1.5pc/[dd] \\
 & \scalea{\gyoung(;;;+,;;;0)} & & \scalea{\gyoung(;;+;*,;;0;*)} & & \scalea{\gyoung(;+;*;*,;0;*;*)} \\
 & \scalea{\gyoung(;;;+,;;;-)} & & \scalea{\gyoung(;;+;*,;;-;*)} & & \scalea{\gyoung(;+;*;*,;-;*;*)}
 }
\]
The categories $\scalea{\gyoung(;+;*;*,;+;*;*)},\scalea{\gyoung(;+;*;*,;0;*;*)},\scalea{\gyoung(;+;*;*,;-;*;*)}$ are defined in \ref{para:paper14appendixanalogues}.  The categories denoted by diagrams of the form  $\scalea{\gyoung(;;;,;;;)}$  can similarly be defined using the Fourier-Mukai transform $\Phi$ instead of $\whPhi$; it  will always be clear from the context as to which functor is used in the definition. A concatenation of more than one such diagram will mean the extension closure of the categories involved; for example, the concatenation
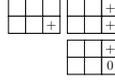
\[
\xymatrix @-2.3pc{
\scalea{\gyoung(;;;,;;;+)} & \scalea{\gyoung(;;;+,;;;+)} \\
& \scalea{\gyoung(;;;+,;;;0)}
}
\]
is the extension closure of all slope semistable fiber sheaves of slope at least zero (including sheaves supported in dimension zero, which are slope semistable fiber sheaves of slope $+\infty$).

\paragraph Take \label{para:phasecomputations} any nonzero object $F \in \Bl$, and suppose $\ch (F)$ is of the form \eqref{eq:chF}. Recall the equations
\begin{align*}
  \oZw (F) &= \tfrac{\omega^2}{2}\ch_1(F) + i(\omega\ch_2(F) - \tfrac{\omega^3}{6}\ch_0(F)) \\
  &= O(v^2)H_B^2\wt{x} + O(1)H_B\wt{S} + i \left( O(v)H_B\wt{\eta} + O(\tfrac{1}{v})\wt{a}-O(v)H_B^2\wt{n} \right) \text{\quad from \eqref{eq:oZwFOvform}}.
\end{align*}

The following phase computations are more streamlined than their counterparts in \cite[4.6]{Lo14}:
\begin{itemize}
\item[(i)] Suppose $F \in \Coh^{\leq 1}(X)$.
  \begin{itemize}
  \item[(a)] If $F \in \Coh^{\leq 0}(X)$, then $\oZw (F)=0$ and $\phi (F) = \tfrac{1}{2}$ by definition.
  \item[(b)] If $\dimension F=1$, then $\omega^2 \ch_1(F)=0$, which means $\oZw (F)$ is purely imaginary and nonzero \cite[Remark 3.3.1]{BMT1}, and so $\phi (F) = \tfrac{1}{2}$.
  \end{itemize}
\item[(ii)] Suppose $F \in  \scalea{\gyoung(;;*;*,;+;*;*)}$.   Then $\wt{x} = f\ch_1(F) > 0$, and so $\phi (F) \to 0$.
\item[(iii)] Suppose $F \in \Coh^2(p)_1$: then $\wt{n}=0$ and $\wt{x}=f\ch_1(F)=0$, and $\ch_1(F)=p^\ast \wt{S}$ is effective.
    \begin{itemize}
    \item[(A)] If $F \in W_{0,\wh{\Phi}}$, then $F \in \scalea{\gyoung(;;+;*,;;+;*)}$, and from the case of $d=2$ in the proof of Lemma \ref{lem:LZ2Remark5-17analogue} we know $\ch_2 (F) p^\ast H_B = H_B \wt{\eta} > 0$ and so $\phi (F) \to \tfrac{1}{2}$.
    \item[(B)] If $F \in W_{1,\wh{\Phi}}$:
      \begin{itemize}
      \item[(a)] If $F \in \scalea{\gyoung(;;+;*,;;-;*)}$, then again from the proof of Lemma \ref{lem:LZ2Remark5-17analogue} we know  $H_B \wt{\eta} < 0$, which gives $\phi (F) \to -\tfrac{1}{2}$.
      \item[(b)] If $F \in \scalea{\gyoung(;;+;*,;;0;*)}$, then $\ch (F) = {\left( \begin{smallmatrix} 0 & \ast & \ast \\ 0 & 0 & \ast \end{smallmatrix}\right)}$ and
          \[
            \oZw (F) = \tfrac{\omega^2}{2}p^\ast \wt{S} + i\omega \wt{a} f = \tfrac{u(hu+2v)}{2} H_B\wt{S} + i u\wt{a}.
          \]
          Since $p^\ast \wt{S}$ is effective, we have $\Re \oZw (F)>0$ and is $O(1)$
          while $\Im \oZw (F)$ is $O(\tfrac{1}{v})$ as $v \to \infty$.  Hence  $\phi (F) \to 0$.  Note that, depending on the sign of $\wt{a}$, we could have $\phi (F) \to 0^+, \phi (F)=0$ or $\phi (F) \to 0^-$.
      \end{itemize}
    \end{itemize}
\item[(iv)] Suppose $F \in \Tc^{l,+}$.  Then $\wt{x}=f\ch_1(F)>0$ and $\phi (F) \to 0$.
\item[(v)] Suppose $F \in \Tc^{l,0}$.  Then $\ch_0(F)\neq 0, \wt{x}=f\ch_1(F)=0$ and $F$ is $\wh{\Phi}$-WIT$_1$ by \ref{para:TFlastdefinitions}(i).  From \ref{para:WIT1shfch1eqzero}, we know that $\oZw(F)$ is dominated by the imaginary part, which is $O(v)$ and negative as $v \to \infty$.  Hence $\phi (F) \to -\tfrac{1}{2}$.
\item[(vi)] Suppose $F=A[1]$ where $A \in \Fc^{l,0}$.  Then as in case (v), $\oZw (A)$ is dominated by its imaginary part,  which is negative as $v \to \infty$.  Hence $\phi (F) = \phi (A[1]) \to \tfrac{1}{2}$.
\item[(vii)] Suppose $F = A[1]$ where $A \in \Fc^{l,-}$.  Then $\wt{x} = f\ch_1(F)>0$ and $\phi (F) \to 0$.
\end{itemize}

\paragraph[The category $W_{1,\wh{\Phi}} \cap \mathcal{T}^{l,0}$] Given any object $M$ in $W_{1,\wh{\Phi}} \cap \Tc^{l,0}$, we know from \ref{para:TFlastdefinitions}(iv) that  $M$ has a filtration $M_0 \subseteq M_1 = M$ in $\Coh (X)$ where $M_0 \in W_{1,\wh{\Phi}} \cap \Coh^{\leq 2}(X)$ and $M_1/M_0 \in \Tc^{l,0}$.  We compute $\phi (M)$ in different cases:
\begin{itemize}
\item Suppose $M_1/M_0 \neq 0$.  Then $M$ is a $\wh{\Phi}$-WIT$_1$ sheaf of nonzero rank with $f\ch_1(F)=0$, meaning $\wh{M} = \whPhi M [1]$ is a coherent sheaf supported in dimension 2.  By \ref{para:WIT1shfch1eqzero}, we know $\Im \oZw (M)$ is negative and $O(v)$ as $v \to \infty$, while $\Re \oZw (M)$ is at most $O(1)$.  Hence $\phi (M) \to  -\tfrac{1}{2}$.
\item Suppose $M_1/M_0=0$.  Then $M=M_0$ lies in $W_{1,\wh{\Phi}} \cap \Coh^{\leq 2}(X)$, and so $M \in \Coh(p)_{\leq 1}$.  Let us write $\ch (M) = {\left( \begin{smallmatrix} 0 & p^\ast \overline{S} & \overline{a}f \\
    0 & \Theta p^\ast \overline{\eta} & \overline{s} \end{smallmatrix}\right)}$.
    \begin{itemize}
    \item If $\dimension M=2$, then by Lemma \ref{lem:LZ2Remark5-17analogue} we have $\ch_2(M)p^\ast H_B = H_B \overline{\eta} \leq 0$.
        \begin{itemize}
        \item If $H_B \overline{\eta} < 0$, then $\phi (M) \to -\tfrac{1}{2}$.
        \item If $H_B \overline{\eta} = 0$, then $\oZw (M) = O(1)H_B \overline{S} + iO(\tfrac{1}{v})\overline{a}$.  Since $H_B \overline{S}$ is the leading coefficient of $\omega^2 \ch_1(M)$ as $v \to \infty$, we have $H_B \overline{S} > 0$, and so $\phi (M) \to 0$.
                \end{itemize}
    \item If $\dimension M=1$, then $M$ is a fiber sheaf and $\phi (M) = \tfrac{1}{2}$.
    \end{itemize}
\end{itemize}

\section{Slope stability vs limit tilt stability}\label{sec:slopestabvslimittilstab}

In this section, we will prove a  comparison theorem between slope stability and limit tilt stability via the Fourier-Mukai transform $\Phi$,  generalising the main result in the product threefold case in the predecessor to this article  \cite[Theorem 5.1]{Lo14}.  Part (A) of  Theorem \ref{thm:paper14thm5-1analogue}  describes a vanishing condition under which a $\mu_\bo$-stable torsion-free sheaf is taken by $\Phi$ to a $\oZl$-stable object, while part (B) of the theorem will explain how a $\oZl$-semistable object can be modified in codimension 2 so that it is taken by $\whPhi$ to a $\mu_\bo$-semistable torsion-free sheaf.  We will then discuss  scenarios under which the  vanishing condition in Theorem \ref{thm:paper14thm5-1analogue}(A) hold.  One of these scenarios is when we have a torsion-free reflexive sheaf.  Since every torsion-free sheaf on a smooth projective threefold differs from its double dual (which is reflexive) only in codimension 2, we can interpret Theorem \ref{thm:paper14thm5-1analogue} as follows: $\mu_\bo$-stability and $\oZl$-stability coincide up to modification in codimension 2 in the derived category.

\paragraph[Positivity of Chern classes]

\begin{lemsub}\label{lem:1dimshch2etaeffective}
Suppose $E \in \Coh^{\leq 1}(X)$ and  $\ch(E) =  {\left( \begin{smallmatrix} 0 & 0 & \ol{a}f \\ 0 & \Theta p^\ast \ol{\eta} & \ol{s}  \end{smallmatrix}\right)}$.  Then $\ol{\eta}$ is an effective divisor class on $B$.
\end{lemsub}

\begin{proof}
Since $E$ is supported in dimension at most 1, the 1-cycle $\Theta p^\ast \ol{\eta} + \ol{a}f$ is effective.  Thus the pushforward $p_\ast (\Theta p^\ast \ol{\eta} + \ol{a}f)$ is also effective.  However, the projection formula gives $p_\ast (\Theta p^\ast \ol{\eta} + \ol{a}f)= (p_\ast \Theta)\ol{\eta}=\ol{\eta}$.  The lemma then follows.
\end{proof}

\begin{lemsub}\label{lem:cohleq10dimfiberpositivity}
For any nonzero $E$ in $\Coh^{\leq 1}(X)$, we have $(p^\ast H_B) \ch_2(E) \geq 0$.  In the event that $(p^\ast H_B) \ch_2(E)=0$ and $E \in \Coh^{\leq 1}(X) \cap \{\Coh^{\leq 0}\}^\uparrow$, the sheaf $E$ is supported in dimension 0 and $\ch_3 (E) > 0$.
\end{lemsub}

\begin{proof}
Take any $E \in \Coh^{\leq 1}(X)$ and suppose $\ch(E) =  {\left( \begin{smallmatrix} 0 & 0 & \ol{a}f \\ 0 & \Theta p^\ast \ol{\eta} & \ol{s}  \end{smallmatrix}\right)}$.  By Lemma \ref{lem:1dimshch2etaeffective}, we know $\ol{\eta}$ is an effective divisor class on $B$, and so
\[
  (p^\ast H_B) \ch_2(E) = (p^\ast H_B) \Theta (p^\ast \ol{\eta}) = H_B \ol{\eta} \geq 0
\]
since $H_B$ is ample on $B$.  Now suppose $(p^\ast H_B) \ch_2(E)=0$ and $E \in \Coh^{\leq 1}(X) \cap \{\Coh^{\leq 0}\}^\uparrow$.  Since $\ol{\eta}$ is effective on $B$, this means $\ol{\eta}=0$ in $A^1(B)$.  Thus $E$ is supported on a finite number of fibers of $p$.  That $E \in \{\Coh^{\leq 0}\}^\uparrow$ then forces $E$ to be supported in dimension 0, and so $\ch_3 (E) > 0$.
\end{proof}

\subparagraph[An auxiliary function $\mu_\ast$ on $\Coh^{\leq 1}(X)$] We \label{para:mulowerastfunctiondefinition} define a function $\mu_\ast$ on $\Coh^{\leq 1}(X)$ by setting, for any $E \in \Coh^{\leq 1}(X)$,
\begin{equation*}
  \mu_\ast (E) = \begin{cases} \tfrac{\ch_3(E)}{(p^\ast H_B) \ch_2(E)} &\text{ if $(p^\ast H_B)\ch_2(E) \neq 0$} \\
  +\infty &\text{ if $(p^\ast H_B)\ch_2(E)=0$}
  \end{cases}.
\end{equation*}
By Lemma \ref{lem:cohleq10dimfiberpositivity}, for any $E \in \Coh^{\leq 1}(X)$ we have $(p^\ast H_B)\ch_2(E) \geq 0$.  Note that if $E \in \Coh^{\leq 1}(X)$ and $(p^\ast H_B)\ch_2(E)=0$, it does not necessarily mean $\ch_3(E)\geq 0$.  For example, for any $E \in  \scalea{\gyoung(;;;+,;;;-)}$ we have $(p^\ast H_B)\ch_2(E)=0$ and $\ch_3(E)< 0$.  Nonetheless, if $E \in \Coh^{\leq 1}(X) \cap W_{0,\Phi}$, then in case of $(p^\ast H_B)\ch_2(E)=0$ we do have $\ch_3(E)\geq 0$: the vanishing $(p^\ast H_B)\ch_2(E)=0$ implies that $E$ is a fiber sheaf as in the proof of Lemma  \ref{lem:cohleq10dimfiberpositivity}.  Moreover, any $\Phi$-WIT$_0$ fiber sheaf lies in the extension closure $\langle \Coh^{\leq 0}(X),  \scalea{\gyoung(;;;+,;;;+)}\rangle$ and so we have  $\ch_3(E) \geq 0$, with equality if and only if $E$ is zero.  (See \ref{para:muloweraststabdefinition} below.)

\subparagraph[$\mu_\ast$-semistability] For \label{para:muloweraststabdefinition} convenience, we will say a sheaf $E \in \Coh^{\leq 1}(X)$ is $\mu_\ast$-semistable if, for every short exact sequence in $\Coh (X)$ of the form
\[
0 \to M \to E \to N \to 0
\]
where $M,N \neq 0$, we have $\mu_\ast (M) \leq \mu_\ast (N)$.  \textit{A priori}, we do not have the Harder-Narasimhan property for $\mu_\ast$ because $(p^\ast H_B)\ch_2(E)=0$ does not guarantee $\ch_3(E) \geq 0$ as mentioned in \ref{para:mulowerastfunctiondefinition}.  Even if we restrict to the category $\Coh^{\leq 1}(X) \cap W_{0,\Phi}$, we cannot apply \cite[Proposition 3.4]{LZ2} directly because this category is not a Serre subcategory of $\Coh (X)$.  However, if we restrict to the case $E \in \Coh^{\leq 1}(X) \cap \{\Coh^{\leq 0}\}^\uparrow$, then the vanishing $(p^\ast H_B)\ch_2(E)=0$ does imply $\ch_3(E)\geq 0$ by Lemma \ref{lem:cohleq10dimfiberpositivity}.  Since $\Coh^{\leq 1}(X) \cap \{\Coh^{\leq 0}\}^\uparrow$ is a Serre subcategory of $\Coh (X)$, and hence a noetherian abelian category in its own right, the function $\mu_\ast$ is a slope-like function on $\Coh^{\leq 1}(X) \cap \{\Coh^{\leq 0}\}^\uparrow$ in the sense of \cite[Sectino 3.2]{LZ2} and has the Harder-Narasimhan property as a result of \cite[Proposition 3.4]{LZ2}.

\subparagraph  We \label{para:muastBdefiition} can also define a `twisted' version of $\mu_\ast$ using the twisted Chern character.  For any $B \in A^1(X)$ we set
\begin{equation*}
  \mu_{\ast,B}(E) = \begin{cases}
    \frac{\ch_3^B(E)}{(p^\ast H_B) \ch_2^B(E)} &\text{ if $(p^\ast H_B)\ch_2^B(E) \neq 0$}\\
    +\infty &\text{ if $(p^\ast H_B)\ch_2^B(E)=0$}
    \end{cases}.
\end{equation*}
  For the specific $B$-field  $B = -\tfrac{1}{2}p^\ast K_B$ and any object $E \in D^b(X)$ with  Chern character $\ch(E) =  {\left( \begin{smallmatrix} 0 & 0 & \ol{a}f \\ 0 & \Theta p^\ast \ol{\eta} & \ol{s}  \end{smallmatrix}\right)}$ where $(p^\ast H_B)\ch_2(E) = H_B \ol{\eta} \neq 0$, we have  $\ch^B(E) = {\left( \begin{smallmatrix} 0 & 0 & \ol{a}f \\ 0 & \Theta p^\ast \ol{\eta}\,\,\, & \ol{s}+\tfrac{1}{2}K_B\ol{\eta}  \end{smallmatrix}\right)}$ and thus
\begin{align*}
  \mu_{\ast,B} (E) &= \frac{\ch_3^B(E)}{(p^\ast H_B) \ch_2^B(E)} \\
  &= \frac{\ol{s} + \tfrac{1}{2}K_B\ol{\eta}}{(p^\ast H_B)\ch_2(E)} \\
  &= \frac{\ol{s}}{H_B\ol{\eta}} + \frac{1}{2}\cdot \frac{hH_B \ol{\eta}}{H_B\ol{\eta}} \\
  &= \frac{\ol{s}}{H_B\ol{\eta}} + \frac{h}{2}.
\end{align*}

\subparagraph Note \label{para:1dimsheaffibercriterion} that, for any $E \in \Coh^{\leq 1}(X)$, we have $\mu_{\ast,B}(E)=\infty$ if and only if $E$ is a fiber sheaf.  This is because
\begin{align*}
  \mu_{\ast,B}(E)= \infty &\Leftrightarrow (p^\ast H_B)\ch_2^B(E) = 0 \text{ from the definition of $\mu_{\ast,B}$} \\
  &\Leftrightarrow H_B \ol{\eta}=0 \text{\quad  (with $\ch(E)$ written as in \ref{para:muastBdefiition})} \\
  &\Leftrightarrow \ol{\eta}=0 \text{ by Lemma \ref{lem:1dimshch2etaeffective}} \\
  &\Leftrightarrow \mathrm{supp}(E) \text{ is contained in a finite union of fibers of $p$},
\end{align*}
i.e.\ $E$ is a fiber sheaf.

We now come to the main theorem of this article.   Theorem \ref{thm:paper14thm5-1analogue}(A) says that every torsion-free sheaf on $X$ satisfying the vanishing condition (V)  are taken by the Fourier-Mukai transform $\Phi$ to a limit tilt stable object.  We will give examples where the condition (V) holds in  Corollary \ref{cor:TodaLinvobjconnection}, Remark \ref{rem:TodaHallAlgconnection} and Corollary \ref{cor:thm1reflexiveexample}.  In particular, Corollary \ref{cor:thm1reflexiveexample} says every torsion-free reflexive sheaf  is taken by $\Phi$ to a limit tilt stable object.

\begin{thm}[$\mu_{\bo,B}$-stability versus $\oZl$-stability]\label{thm:paper14thm5-1analogue}
Let $p : X \to B$ be a Weierstra{\ss} elliptic threefold satisfying assumptions (1)-(3) in \ref{para:ellipfibdef}.  Fix  any $a,b>0$ satisfying \eqref{eq:constraint1}, and fix the $B$-field $B=-\tfrac{1}{2}p^\ast K_B$.
\begin{itemize}
\item[(A)] Suppose $E$ is a  $\mu_{\bo,B}$-stable torsion-free sheaf $E$ on $X$ satisfying the following  condition:
\begin{itemize}
\item[(V)] Given any $T \in \Coh^{\leq 1}(X) \cap W_{0,\Phi}$ satisfying
\begin{equation}\label{eq:vanishingV2eqn}
  \mu_{\ast,B}(T) \geq\frac{2}{a(ha+2b)H_B^2} \cdot \mu_{\bo,B} (E)
\end{equation}
that comes with some $\alpha \in \Hom (T,E[1])$ such that $\Phi \alpha$ is a $\Bl$-injection, $T$ must be zero.
\end{itemize}

Then $\Phi E[1] \in \Bl$ is a $\oZl$-stable object.
\item[(B)] Suppose $F \in \Bl$ is a $\oZl$-semistable object with $f\ch_1(F) \neq 0$.  Consider the $\Bl$-short exact sequence
\[
  0 \to F' \to F \to F'' \to 0
\]
where $F' \in \langle \Fc^l[1], W_{0,\whPhi}\rangle$ and $F'' \in W_{1,\whPhi} \cap \Tc^l$, which exists by the torsion triple \eqref{eq:Bltorsiontriple1}.  Then $\whPhi F'$ is a torsion-free $\mu_{\bo,B}$-semistable sheaf, and $\whPhi F'' [1] \in \Coh^{\leq 1}(X)$.
\end{itemize}
\end{thm}

In the proof of Theorem \ref{thm:paper14thm5-1analogue}(A), we will often consider the twisted slope function $\mu_{\bo,B}$ instead of $\mu_\bo$.

\subparagraph[Overview of the rest of this section] The rest of this section is devoted to the proof of Theorem \ref{thm:paper14thm5-1analogue}, as well as  examples of sheaves and objects that satisfy the conditions in Theorem \ref{thm:paper14thm5-1analogue}.  The strategy of the proof of Theorem \ref{thm:paper14thm5-1analogue}(A) is simple: Given  a $\mu_{\bo,B}$-stable torsion-free sheaf $E$ satisfying the vanishing condition (V), we consider a $\Bl$-subobject $G$ of $F$.  We then show that $\phi (G) \prec \phi (F/G)$ except in the case where $\dimension \wh{G} \leq 1$, which can be ruled out by the vanishing condition.

After we present the proof of Theorem \ref{thm:paper14thm5-1analogue}(A), we give some examples of slope semistable sheaves that satisfy the vanishing condition (V).  These include torsion-free reflexive sheaves (Corollary \ref{cor:thm1reflexiveexample}).  That is, given any torsion-free reflexive sheaf $E$ on a Weierstra{\ss} elliptic threefold, even though its transform $\Phi E$ may be  a 2-term complex in general, but it is always a $\oZl$-stable (i.e.\ limit tilt stable) object in $\Bl$.

In Corollary \ref{cor:TodaLinvobjconnection} and Remark \ref{rem:TodaHallAlgconnection}, we point out that given a variation of Pandharipande-Thomas' stable pairs that appears in Toda's higher-rank DT/PT correspondence formula \cite{Toda2}, the left cohomology of the stable pair is always a slope-semistable torsion-free sheaf that satisfies the vanishing condition (V).  This means that given a  stable pair $E \in D^b(X)$ considered in \cite{Toda2},  if we modify $E$ in codimension 2, then $E$ is taken by $\Phi$ to a $\oZl$-semistable object.

As for the proof of Theorem \ref{thm:paper14thm5-1analogue}(B), its outline is the same as that of  \cite[Theorem 5.1B]{Lo14}, while the technical ingredients are given in Lemmas \ref{lem:paper14lem5-4analogue} through \ref{lem:paper14lem5-9analogue}.  The strategy of proof for part (B) is as follows: Having written a $\oZl$-semistable object $F$ as an extension of an  object $F''$ by another object $F'$, we show that $\whPhi$ takes $F''$ to a torsion sheaf supported in codimension at least 2, while $\whPhi$ takes $F'$ to a slope semistable torsion-free sheaf.

\begin{proof}[Proof of Theorem \ref{thm:paper14thm5-1analogue}(A)]
We follow the outline of the proof of \cite[Theorem 5.1(A)]{Lo14}.  Let $E$ be as described in the hypotheses, and write $F = \Phi E[1]$.    Since $\ch_0(E)>0$, we have $f\ch_1(F)>0$, and so $\phi (F) \to 0$.  We have $F \in \Bl$ from \ref{para:TFlastdefinitions}(v).

Take any $\Bl$-short exact sequence of the form
\begin{equation}
0 \to G \to F \to F/G \to 0
\end{equation}
where $G \neq 0$.  This gives an exact sequence of sheaves
\begin{equation}\label{eq:les1}
0 \to \wh{\Phi}^0 G \to E \overset{\alpha}{\to} \wh{\Phi}^0 (F/G) \to \wh{\Phi}^1G \to 0
\end{equation}
and $\wh{\Phi}^1 (F/G)=0$.  From the torsion triple \eqref{eq:Bltorsiontriple1} in $\Bl$, we also have the decomposition of $G$ by an exact triangle
\[
  \Phi (\wh{\Phi}^0 G)[1] \to G \to \Phi (\wh{\Phi}^1 G) \to \Phi (\wh{\Phi}^0 G) [2]
\]
where $\Phi (\wh{\Phi}^0 G)[1] \in \langle \Fc^l [1], W_{0,\wh{\Phi}} \rangle$ and $\Phi (\wh{\Phi}^1 G) \in W_{1,\wh{\Phi}} \cap \Tc^l$.

\textbf{Suppose $\ch_0 (\image \alpha) =0$.}  Then $\ch_0 (\wh{\Phi}^0 G) = \ch_0(E)>0$, giving $f\ch_1 (\Phi (\wh{\Phi}^0 G)[1]) > 0$.  We divide further into two cases:
\begin{itemize}
\item[(a)] $\ch_1 (\image \alpha) \neq 0$.  In this case, $\mu_{\bo,B} (\wh{\Phi}^0G) < \mu_{\bo,B}(E)$, which gives $\phi (\Phi (\wh{\Phi}^0 G)[1]) \prec \phi (F)$.  On the other hand, we know $\Phi (\wh{\Phi}^1 G) \in W_{1,\wh{\Phi}} \cap \Tc^l$ from above.  By \ref{para:TFlastdefinitions}(iv), we have $\whPhi^1 G \in \Coh^{\leq 2}(X)$.
    \begin{itemize}
    \item[(i)] If $\dimension \whPhi^1 G = 2$, then $\bo \ch_1^B(\whPhi^1 G) = \bo \ch_1 (\whPhi^1 G) > 0$.  From the computation in \ref{eq:ImoZwPhiE1bo2ch1BEcomparison} we know $\Im \oZw (\Phi (\whPhi^1 G)[1])$ is $O(v)$ and is positive as $v \to \infty$.  Thus $\Im \oZw (\Phi (\whPhi^1 G))$ is $O(v)$ and is negative as $v \to \infty$.  Overall, we have $\phi (G) \prec \phi (F)$.
    \item[(ii)] If $\dimension \whPhi^1 G \leq 1$, then $\ch (\whPhi^1 G) = {\left( \begin{smallmatrix} 0 & 0 & \ast \\ 0 & \ast & \ast \end{smallmatrix}\right)}$ and so  $\ch (\Phi (\whPhi^1 G)) = {\left( \begin{smallmatrix} 0 & \ast & \ast \\ 0 & 0 & \ast \end{smallmatrix}\right)}$.  This means that $\oZw (\Phi (\whPhi^1 G))$ itself has order of magnitude at most $O(1)$ as $v \to \infty$ while $\Phi (\whPhi^0 G)[1]$ has order of magnitude $O(v^2)$.  Overall, we have $\phi (G) \prec \phi (F)$.
    \end{itemize}
\item[(b)] $\ch_1 (\image \alpha) =0$.  Then $\image \alpha \in \Coh^{\leq 1}(X)$, and $\ch_0^B, \ch_1^B$ of $\whPhi^0 G, E$ agree, and so $\Phi (\whPhi^0 G)[1], F$ agree in the following components of $\ch$ marked with $\ast$'s: ${\left( \begin{smallmatrix} \ast & \cdot & \cdot \\ \ast & \ast & \cdot \end{smallmatrix}\right)}$.  As a consequence, the central charges $\oZw (\Phi (\whPhi^0 G)[1]), \oZw (F)$ agree in the terms with $O(v^2)$ and  $O(v)$  coefficients, which are the highest orders of magnitudes that occur in $\oZw$.

    On the other hand, we have $\whPhi^1 G \in \Coh^{\leq 2}(X)$ as in (a).
    \begin{itemize}
    \item[(i)] If $\dimension \whPhi^1 G = 2$, then by the same argument as in (a)(i), we know $\oZw (\Phi (\whPhi^1 G))$ is $O(v)$, and the $O(v)$ terms are all part of $\Im \oZw (\Phi (\whPhi^1 G))$, which is negative for $v \gg 0$.  As a result, $\phi (G) \prec \phi (F)$.
    \item[(ii)] If $\dimension \whPhi^1 G \leq 1$, then $\whPhi^0 (F/G)$ also lies in $\Coh^{\leq 1}(X)$.  Thus $\whPhi^0 (F/G)$ is a $\Phi$-WIT$_1$ sheaf in $\Coh^{\leq 1}(X)$, forcing $\whPhi^0 (F/G)$ to be a $\Phi$-WIT$_1$ fiber sheaf \cite[Remark 3.14, Lemma 3.15]{Lo11}.  Hence $F/G$ is a $\whPhi$-WIT$_0$ fiber sheaf, and $\phi (F/G)=\tfrac{1}{2}$.  On the other hand, since $f\ch_1 (G)=f\ch_1 (F)>0$, we have $\phi (G) \to 0$.  Hence $\phi (G) \prec \phi (F/G)$ overall.
    \end{itemize}
\end{itemize}

\textbf{Suppose $\ch_0 (\image \alpha) > 0$.} If $\whPhi^0 G \neq 0$, then $0 < \ch_0 (\whPhi^0 G) < \ch_0 (E)$, which gives $\mu_{\bo,B}(\whPhi^0 G) < \mu_{\bo,B} (E)$.  We still have $\whPhi^1 (G) \in \Coh^{\leq 2}(X)$, and the same argument as in part (a) above would give $\phi (G) \prec \phi (F)$.  Therefore, let us  assume $\whPhi^0 G=0$ from now on, in which case the exact sequence \eqref{eq:les1} reduces to
\[
0 \to E \to \whPhi^0 (F/G) \to \whPhi^1 G \to 0.
\]
Since $\whPhi^0 G=0$, we have $G = \Phi (\whPhi^1 G) \in W_{1,\whPhi} \cap \Tc^l$ and $\wh{G} = \whPhi G [1] \in \Coh^{\leq 2}(X) \cap W_{0,\Phi}$.
\begin{itemize}
\item[(i)] If $\dimension \wh{G}=2$, then $\bo^2 \ch_1^B(\wh{G})>0$ and the computation in (a)(i) gives $\phi (G) \to -\tfrac{1}{2}$, which implies $\phi (G) \prec \phi (F)$.
\item[(ii)] If $\dimension \wh{G} \leq 1$, let us write $\ch (\wh{G}) = {\left( \begin{smallmatrix} 0 & 0 & \ol{a}f \\ 0 & \Theta p^\ast \ol{\eta} & \ol{s}\end{smallmatrix}\right)}$.  Then $\ch^B(\wh{G}) = {\left( \begin{smallmatrix} 0 & 0 & \ol{a}f \\ 0 & \Theta p^\ast \ol{\eta}\,\,\, & \ol{s}+\tfrac{1}{2}K_B\ol{\eta}\end{smallmatrix}\right)}$ and hence
     $\ch (G) = {\left( \begin{smallmatrix} 0 & p^\ast \ol{\eta}\,\,\, & (\ol{s}+\tfrac{1}{2}K_B\ol{\eta})f \\ 0 & 0 & -\ol{a}\end{smallmatrix}\right)}$. We consider the following two possibilities:

    Case (1):  $p^\ast \ol{\eta}=0$.    Then $G$ is a $\whPhi$-WIT$_1$ fiber sheaf  and $\phi (G) = \tfrac{1}{2} \succ \phi (F)$.  In this case
    \begin{equation}\label{eq:mainthmAHomisomstrings}
    \Hom (G,F)\cong\Hom (\Phi \wh{G}, \Phi E[1]) \cong \Hom (\wh{G},E[1])
     \end{equation}
     where $\wh{G}$ is a $\Phi$-WIT$_0$ fiber sheaf, in which case $\mu_{\ast,B}(\wh{G})=\infty$.  Thus the vanishing condition (V) prevents this case from occurring since we are assuming $G$ is nonzero.

    Case (2):  $p^\ast \ol{\eta} \neq 0$.  Then $p^\ast\ol{\eta} = \ch_1 (G)$ is effective since $G$ is a coherent sheaf.  Since $\wh{G} \in \Coh^{\leq 1}(X)$, the proof of Lemma \ref{lem:cohleq10dimfiberpositivity} gives $H_B \ol{\eta} > 0$.  Now
    \begin{align}
      \oZw (G) &= \tfrac{\omega^2}{2}\ch_1(G) + i\omega \ch_2 (G) \notag\\
       &=   \tfrac{u(hu+2v)}{2}H_B\ol{\eta} + i    u(\ol{s} + \tfrac{1}{2}K_B\ol{\eta}) \notag \\
      &= uvH_B\ol{\eta} + O(\tfrac{1}{v^2})  + i u(\ol{s} + \tfrac{1}{2}K_B\ol{\eta}) \notag\\
      &= uH_B\ol{\eta} \left( v + O(\tfrac{1}{v}) + i \cdot \tfrac{\ol{s}+(1/2)hH_B\ol{\eta}}{H_B\ol{\eta}} \right) \notag\\
      &=    uH_B\ol{\eta} \left( v + O(\tfrac{1}{v}) + i \left( \tfrac{\ol{s}}{H_B\ol{\eta}} + \tfrac{h}{2}\right)\right) \notag\\
      &=    uH_B\ol{\eta} \left( v + O(\tfrac{1}{v}) + i \mu_{\ast,B}(\wh{G})\right). \label{eq:thm1-oZwG}
    \end{align}

    On the other hand,
    \begin{equation*}
      \oZw (F) = \oZw (\Phi E [1])
      = \Re \oZw (\Phi E [1]) + i \Im \oZw (\Phi E [1])
    \end{equation*}
    where, for $v \gg 0$ under the constraint \eqref{eq:limitcurve},
    \begin{align*}
      \Re \oZw (\Phi E [1]) &= \tfrac{v^2}{2}H_B^2 \ch_0(E) + O(1) \text{\quad from \ref{eq:oZwFcomplete}}, \\
      \Im \oZw (\Phi E [1]) &= \left( \tfrac{(\omega^3/6)}{\Theta \bo^2 }\cdot \bo^2 \ch_1^B(E) - \omega_1 A_3 \right) \text{\quad from \ref{eq:ImoZwPhiE1bo2ch1BEcomparison}} \\
      &= \left( \tfrac{(\omega^3/6)}{\Theta \bo^2 }\cdot \bo^2 \ch_1^B(E) - O(\tfrac{1}{v}) \right)
    \end{align*}
We also have $\Theta \bo^2 = (ha+b)^2H_B^2$ from \ref{para:solveChoweq} while
    \begin{align*}
    \tfrac{\omega^3}{6} &= \tfrac{1}{6}u(h^2u^2 +3huv+3v^2)H_B^2 \text{\quad also from \ref{para:solveChoweq}} \\
    &= \tfrac{(hu+v)(ha+b)^2}{a(ha+2b)}H_B^2 \text{\quad by \eqref{eq:limitcurve}},
    \end{align*}
    giving us
    \[
    \tfrac{(\omega^3/6)}{\Theta \bo^2 } = \tfrac{hu+v}{a(ha+2b)}.
    \]
    Overall,
    \begin{align}
      \oZw (F) &= \tfrac{v^2}{2}H_B^2\ch_0(E) + O(1) + i \left( \tfrac{1}{a(ha+2b)}\bo^2\ch_1^B(E)v + O(\tfrac{1}{v})\right) \notag\\
      &= \tfrac{1}{2}H_B^2\ch_0(E) \left( v^2 + O(1) + i \left( \tfrac{2}{a(ha+2b)H_B^2} \mu_{\bo,B} (E) v + O(\tfrac{1}{v})\right)\right). \label{eq:thm1-oZwF}
    \end{align}
    Comparing \eqref{eq:thm1-oZwG} with \eqref{eq:thm1-oZwF} and using the string of isomorphisms \eqref{eq:mainthmAHomisomstrings}, we see that the vanishing  (V) ensures we  have $\mu_{\ast,B}(\wh{G}) < \tfrac{2}{a(ha+2b)H_B^2} \mu_{\bo,B} (E)$ and hence  $\phi (G) \prec \phi (F)$ in this case.
\end{itemize}
Hence $F$ is $\oZl$-stable.
\end{proof}

\begin{remsub}
One can reasonably expect to obtain a finer version of Theorem \ref{thm:paper14thm5-1analogue}(A) by studying the $O(\tfrac{1}{v^2})$ term in $\oZw (G)$ and the $O(v)$ term in $\oZw (F)$ in its proof more closely, although we will not pursue this in this article.
\end{remsub}

\paragraph For \label{para:1dimEslopefunction} any polarisation $\bo$ on $X$ and any $B$-field $B \in A^1(X)$, we will consider another slope function $\ol{\mu}_{\bo,B}$ on $\Coh^{\leq 1}(X)$ by setting, for any $E \in \Coh^{\leq 1}(X)$,
\[
  \ol{\mu}_{\bo,B} (E) = \begin{cases}
  \tfrac{\ch_3^B(E)}{\bo\ch_2^B(E)} &\text{ if $\bo\ch_2^B(E)\neq 0$} \\
  \infty &\text{ if $\bo\ch_2^B(E)=0$}
  \end{cases}.
\]
It is clear that $\ol{\mu}_{\bo,B}$ is a slope function on $\Coh^{\leq 1}(X)$ with the Harder-Narasimhan property.  Let us build on the notation in \cite[Section 3.3]{Toda2} and define
\[
\mathcal{C}^B_{(0,\infty]} = \langle A \in \Coh^{\leq 1}(X): \text{ $A$ is $\ol{\mu}_{\bo,B}$-semistable and $\ol{\mu}_{\bo,B}(A) \in (0,\infty]$}\rangle.
\]
When the $B$-field is zero, we will simply write $\mathcal{C}_{(0,\infty]}$ for $\mathcal{C}^B_{(0,\infty]}$.

\begin{lemsub}\label{lem:C0inftyinWIT0Phi}
For any $B \in A^1(X)$ we have $\mathcal{C}^B_{(0,\infty]} \subseteq \Coh^{\leq 1}(X) \cap W_{0,\Phi}$.
\end{lemsub}

\begin{proof}
Take any $E \in \mathcal{C}^B_{(0,\infty]}$ and consider the short exact sequence in $\Coh (X)$
\[
 0 \to E' \to E \to E'' \to 0
\]
where $E'$ is the maximal subsheaf lying in $\{\Coh^{\leq 0}\}^\uparrow$.  Then  $E''$ is necessarily  a fiber sheaf all of whose $\ol{\mu}_{\bo,B}$-HN factors have $\ch_3^B =\ch_3>0$ (note that $\ch_3^B=\ch_3$ for fiber sheaves).  This implies $E''$ is $\Phi$-WIT$_0$ \cite[Proposition 6.38]{FMNT}.  On the other hand, $E'$ is $\Phi$-WIT$_0$ by \cite[Remark 3.14]{Lo11}.  Hence $E$ itself is $\Phi$-WIT$_0$.
\end{proof}

\begin{lem}\label{para:vanishingV1plus2variant1}
 Suppose $a,b>0$ satisfy \eqref{eq:constraint1} and we take the $B$-field to be $B=-\tfrac{1}{2}p^\ast K_B$.  Suppose $E$ is a torsion-free sheaf on $X$ with $\mu_{\bo,B}(E)>0$.  Then
\begin{align*}
\{ T \in \Coh^{\leq 1}(X) \cap W_{0,\Phi} : \, &T \text{ satisfies \eqref{eq:vanishingV2eqn}}\} \\
&\subseteq
\{ T \in \Coh^{\leq 1}(X) \cap W_{0,\Phi} : \mu_{\ast,B} (T) > 0 \} \\
&= \{ T \in \Coh^{\leq 1}(X) \cap W_{0,\Phi} : \text{ $T$ is a fiber sheaf or $\ch_3^B(T)>0$}\} \\
&= \{ T \in \Coh^{\leq 1}(X) \cap W_{0,\Phi} : \ol{\mu}_{\bo,B}(T)>0\}.
\end{align*}
\end{lem}

\begin{proof}
The inclusion in the lemma is clear.  To see why the first equality  is true, take any $T \in \Coh^{\leq 1}(X) \cap W_{0,\Phi}$.
\begin{itemize}
\item If $T$ is a fiber sheaf, then $\mu_{\ast,B}(T)=\infty$; if $T$ is not a fiber sheaf but $\ch_3^B(T)>0$, then $(p^\ast H_B)\ch_2^B(T)>0$ by \ref{para:1dimsheaffibercriterion} and hence $\mu_{\ast,B}(T)>0$.  This gives the inclusion $\supseteq$ in the first equality in the lemma.
\item If $\mu_{\ast,B}(T)>0$ and  $T$ is not a fiber sheaf, then  $(p^\ast H_B)\ch_2^B(T)>0$ by \ref{para:1dimsheaffibercriterion} as before, and we have $\ch_3^B(T)>0$.  This gives the inclusion $\subseteq$ in the first equality.
\end{itemize}

To see why the second equality  is true, take any $0 \neq T \in \Coh^{\leq 1}(X) \cap W_{0,\Phi}$.
\begin{itemize}
\item If $T$ is a 0-dimensional fiber sheaf, then $\ol{\mu}_{\bo,B}(T)=\infty$.  If $T$ is a 1-dimensional fiber sheaf, then $\ch_3^B(T)=\ch_3(T)>0$ because $T$ is $\Phi$-WIT$_0$ \cite[Proposition 6.38]{FMNT}, and hence $\ol{\mu}_{\bo,B}(T)>0$.  If $T$ is not a fiber sheaf but $\ch_3^B(T)>0$, then $T$ must be supported in dimension 1 and $\ol{\mu}_{\bo,B}(T)>0$.  This proves the inclusion $\subseteq$ in the second equality in the lemma.
\item If $\ol{\mu}_{\bo,B}(T)>0$ and $T$ is not a fiber sheaf, then $T$ must be supported in dimension 1, i.e.\ $\bo \ch_2^B(T)>0$, giving us $\ch_3^B(T)>0$.  This proves the inclusion $\supseteq$ in the second equality.
\end{itemize}
\end{proof}

\begin{remsub}\label{rem:VprimeimpliesV}
By Lemma \ref{para:vanishingV1plus2variant1}, if $E$ is a torsion-free sheaf on $X$ with $\mu_{\bo,B}(E)>0$, then the vanishing condition     (V)  in Theorem \ref{thm:paper14thm5-1analogue}(A) follows from the  vanishing condition
\begin{itemize}
\item[(V')] Given any $T \in  \Coh^{\leq 1}(X)\cap W_{0,\Phi}$ satisfying $\ol{\mu}_{\bo,B}(T)>0$, that comes with some morphism $\alpha \in \Hom (T,E[1])$ such that $\Phi \alpha$ is a $\Bl$-injection, $T$ must be zero.
\end{itemize}
\end{remsub}

\begin{cor}\label{cor:TodaLinvobjconnection}
Let $p : X \to B$ be a Weierstra{\ss} elliptic threefold satisfying assumptions (1)-(3) in \ref{para:ellipfibdef}.  Fix  any $a,b>0$ satisfying \eqref{eq:constraint1}, and fix the $B$-field $B=-\tfrac{1}{2}p^\ast K_B$.  Then for any $\mu_{\bo,B}$-stable torsion-free sheaf $E$ satisfying $\mu_{\bo,B}(E)>0$ and the vanishing
\begin{equation}\label{eq:paper14thm5-1Aanacor1vanishing}
  \Hom (\mathcal{C}^B_{(0,\infty]},E[1])=0,
\end{equation}
the object $\Phi E [1]$ lies in $\Bl$ and is  $\oZl$-stable.
\end{cor}

\begin{proof}
By Remark \ref{rem:VprimeimpliesV} and  Theorem \ref{thm:paper14thm5-1analogue}(A), it suffices to  show that the vanishing \eqref{eq:paper14thm5-1Aanacor1vanishing} implies the vanishing (V') for any  $E$ as described.

Take any $T \in  \Coh^{\leq 1}(X)\cap W_{0,\Phi} $ with $\ol{\mu}_{\bo,B}(T) > 0$, and let $T'$ be the left-most $\ol{\mu}_{\bo,B}$-HN factor of $T$.  Then $T'$, if nonzero,  is a  $\ol{\mu}_{\bo,B}$-semistable sheaf with $\ol{\mu}_{\bo,B}(T')>0$, and so $T'$ lies in $\mathcal{C}_{(0,\infty]}^B$.  From Lemma \ref{lem:C0inftyinWIT0Phi}, we know $T'$ is $\Phi$-WIT$_0$.  Thus all the terms are $\Phi$-WIT$_0$ in the short exact sequence of sheaves
\[
0 \to T' \overset{\gamma}{\to} T \to T/T' \to 0,
\]
and applying $\Phi$ gives the exact sequence of sheaves in $\Coh (p)_{\leq 1}$
\[
0 \to \wh{T'} \overset{\Phi \gamma}{\longrightarrow} \wh{T} \to \wh{T/T'} \to 0
\]
which is also a $\Bl$-short exact sequence since $\Coh (p)_{\leq 1}$ is contained in $\Bl$ by \ref{para:TlFlbasicproperties}(i).  In particular, $\wh{T'}$ is a $\Bl$-subobject of $\wh{T}$.

Now suppose we have a morphism $\alpha \in \Hom (T,E[1])$ such that $\Phi \alpha$ is a $\Bl$-injection.  From the previous paragraph, we have $\Bl$-injections
\[
  \wh{T'} \overset{\Phi\gamma}{\longrightarrow} \wh{T} \overset{\Phi \alpha}{\longrightarrow} \Phi E [1].
\]
The composition $(\Phi \alpha)(\Phi \gamma)=\Phi (\alpha \gamma)$ is a morphism $\wh{T'} \to \Phi E[1]$, i.e.\ $\alpha \gamma$ is a morphism $T' \to E[1]$, which must be zero by the vanishing \eqref{eq:paper14thm5-1Aanacor1vanishing}.  This implies that $\wh{T'}$ must be zero, i.e.\ $T'$ must be zero, i.e.\ $T$ must be zero.  Hence $E$ satisfies the vanishing (V') as desired, finishing the proof of this corollary.
\end{proof}

\begin{rem}\label{rem:TodaHallAlgconnection}
If we replace $\mathcal{C}^B_{(0,\infty]}$ with $\mathcal{C}^B_{[0,\infty]}$ (the additional objects being $\ol{\mu}_{\bo,B}$-semistable sheaves $A$ with $\ol{\mu}_{\bo,B}(A)=0$) and choose the $B$-field to be zero, then the vanishing condition \eqref{eq:paper14thm5-1Aanacor1vanishing} is satisfied by the left cohomology $H^{-1}(E)$, for any $E$ lying in the category $\mathrm{Coh}^L_\mu (X)$ studied by Toda \cite[Section 4]{Toda2}.  The objects $E$ in $\mathrm{Coh}^L_\mu (X)$ are 2-term complexes that give rise to  `$L$-invariants', and their right cohomology $H^0(E)$ lie in $\mathcal{C}_{[0,\infty]}$.
\end{rem}

\begin{cor}\label{cor:thm1reflexiveexample}
Let $p : X \to B$ be a Weierstra{\ss} elliptic threefold satisfying assumptions (1)-(3) in \ref{para:ellipfibdef}.  Fix  any $a,b>0$ satisfying \eqref{eq:constraint1}, and fix the $B$-field $B=-\tfrac{1}{2}p^\ast K_B$.  Suppose  $E$ is a $\mu_{\bo,B}$-stable (torsion-free) reflexive sheaf.  Then $\Phi E [1]$ lies in $\Bl$ and is a $\oZl$-stable object.
\end{cor}

\begin{proof}
For any torsion-free reflexive sheaf $E$ on a threefold $X$, we have the vanishing
\[
\Hom (\Coh^{\leq 1}(X), E[1])=0
\]
by \cite[Lemma 4.20]{CL}, and so $E$ satisfies  the vanishing condition (V).
\end{proof}

\paragraph In preparation for proving Theorem \ref{thm:paper14thm5-1analogue}(B), we need to understand $\oZl$-semistable objects a little more.  All the following lemmas have appeared in \cite[Section 5]{Lo14} before, where they were proved for the case of the product elliptic threefold.

\begin{lemsub}\label{lem:paper14lem5-4analogue}
$\Coh^{\leq 0}(X), \Coh (p)_0, \Coh^{\leq 1}(X)$ are all Serre subcategories of $\Bl$.
\end{lemsub}

\begin{proof}
By parts (1)(a) and (2)(a) of Lemma \ref{lem:4-1n4-3paper14analogue}, we can interpret $\Bl$ as the category of all $E \in D^b(X)$ such that $E \in \Bc_\omega$ for all $v \gg 0$ while $u,v>0$ are subject to the constraint \eqref{eq:limitcurve}.  Since $\Coh^{\leq 0}(X), \Coh (p)_0, \Coh^{\leq 1}(X)$ are all Serre subcategories of $\Bc_\omega$ for any ample class $\omega$, the lemma follows.
\end{proof}

\begin{lemsub}\label{lem:paper14lem5-5analogue}
The category $\langle \Coh^{\leq 1}(X), \Fc^{l,0}[1]\rangle$ is a torsion class in $\Bl$.
\end{lemsub}

\begin{proof}
The proof of \cite[Lemma 5.5]{Lo14} carries over, with Lemma \ref{lem:paper14lem5-4analogue} playing the role of \cite[Lemma 5.4]{Lo14}.
\end{proof}

We now define
\begin{align*}
  \Ac_\bullet &= \Coh^{\leq 0}(X) \\
  \Ac_{\bullet,1/2} &= \langle \Coh^{\leq 1}(X), W_{0,\whPhi} \cap \Coh (p)_{\leq 1}, \Fc^{l,0}[1]\rangle \\
  &= \langle
\xymatrix @-2.3pc{
\scalea{\gyoung(;;;,;;;+)} & \scalea{\gyoung(;;;+,;;;+)} & \scalea{\gyoung(;;;*,;;+;*)} & \scalea{\gyoung(;;+;*,;;+;*)} \\
& \scalea{\gyoung(;;;+,;;;0)} & &  \\
& \scalea{\gyoung(;;;+,;;;-)} & &
}, \Fc^{l,0}[1]\rangle.
\end{align*}

\begin{lemsub}\label{lem:paper14lem5-7analogue}
$\Ac_{\bullet,1/2}$ is closed under quotient in $\Bl$, and every object in this category satisfies $\phi \to \tfrac{1}{2}$.
\end{lemsub}

\begin{proof}
The lemma follows from the same argument as in the proof of \cite[Lemma 5.7]{Lo14}, with $\ch_{10}$ replaced with $f\ch_1$,  Lemma \ref{lem:paper14lem5-5analogue} playing the role of \cite[Lemma 5.5]{Lo14}, and \ref{para:TlFlbasicproperties}(iv) playing the role of \cite[Remark 4.4(iv)]{Lo14}; the phase computations needed in the argument come from \ref{para:phasecomputations}.
\end{proof}

\begin{lemsub}\label{lem:paper14lem5-8analogue}
Suppose $F \in \Bl$ is a $\oZl$-semistable object and $f\ch_1(F) > 0$.  Then the vanishing $\Hom (\Coh^{\leq 2}(X),\whPhi F)=0$ holds.  In particular, if $F=\Phi E[1]$ for some coherent sheaf $E$, then $E$ is a torsion-free sheaf.
\end{lemsub}

\begin{proof}
The argument in the proof of \cite[Lemma 5.8]{Lo14} still works, with $\ch_{10}$ replaced with $f\ch_1$, and using \ref{para:TFlastdefinitions}(iv) and Lemma \ref{lem:paper14lem5-7analogue} instead of \cite[Lemma 4.10]{Lo14} and \cite[Lemma 5.7]{Lo14}, respectively.
\end{proof}

\begin{lemsub}\label{lem:paper14lem5-9analogue}
Suppose $F \in \Bl$ is $\oZl$-semistable and $f\ch_1 (F) > 0$.  Let $F_1$ denote the $\whPhi$-WIT$_1$ part of $H^0(F)$.  Then $\wh{F_1} = \whPhi F_1 [1]$ lies in $\Coh^{\leq 1}(X)$.
\end{lemsub}

\begin{proof}
The same argument as in the proof of \cite[Lemma 5.9]{Lo14} shows $F_1 \in W_{1,\whPhi} \cap \Coh (p)_{\leq 1}$, and that we have a $\Bl$-surjection $F \twoheadrightarrow F_1$.  Let us write $\ch (F_1) = {\left( \begin{smallmatrix} 0 & p^\ast \ol{S} & \ol{a}f \\ 0 & \Theta p^\ast \ol{\eta} & \ol{s} \end{smallmatrix}\right)}$.  By Lemma \ref{lem:LZ2Remark5-17analogue}, we have $\Theta p^\ast (H_B\ol{\eta}) = H_B\ol{\eta}\leq 0$.

If $H_B\ol{\eta} < 0$, then  $\phi (F_1) \to -\tfrac{1}{2}$, meaning $F_1$ destabilises $F$.  Hence $H_B\ol{\eta}=0$ must hold.  On the other hand, since $F_1$ is $\whPhi$-WIT$_1$,  it follows that $p^\ast \ol{\eta}$ is an effective divisor class on $X$.  Now, for any ample divisor of the form $\Theta + vp^\ast H_B$ we have
\begin{align*}
  (\Theta + vp^\ast H_B)^2 p^\ast \ol{\eta} &= \Theta p^\ast (K_B\ol{\eta}) + 2v \Theta p^\ast (H_B\ol{\eta}) \\
  &= (h+2v)\Theta p^\ast (H_B\ol{\eta}) \\
  &=0,
\end{align*}
which means $p^\ast \ol{\eta}=0$ in $A^1(X)$.  Thus $\ch (\whPhi F_1) = {\left( \begin{smallmatrix} 0 & 0 & \ast \\ 0 & \ast & \ast \end{smallmatrix}\right)}$, i.e.\ $\whPhi F_1 [1]$ lies in $\Coh^{\leq 1}(X)$ as claimed.
\end{proof}

\begin{proof}[Proof of Theorem \ref{thm:paper14thm5-1analogue}(B)]
The proof of \cite[Theorem 5.1B]{Lo14} still applies, so we merely give an outline here:  In the decomposition $0 \to F' \to F \to F'' \to 0$, the transform $\whPhi F'$ is a coherent sheaf sitting at degree 0 by construction.  Since $F$ is $\oZl$-semistable, we have the vanishing $\Hom (\Coh^{\leq 2}(X), \whPhi F)=0$ from Lemma \ref{lem:paper14lem5-8analogue}, which implies $\Hom (\Coh^{\leq 2}(X), \whPhi F')=0$, proving $\whPhi F'$ is a torsion-free sheaf.

To see that $\whPhi F'$ is $\mu_{\bo,B}$-semistable, take any short exact sequence of sheaves $0 \to M \to \whPhi F' \to N \to 0$ where both $M, N$ are torsion-free.  One can then show that $\Phi M [1]$ is a $\Bl$-subobject of $F$, and that the $\oZl$-semistability of $F$ implies $\mu_{\bo,B}(M) \leq \mu_{\bo,B} (\whPhi F)=\mu_{\bo,B}(\whPhi F')$, where the last equality uses the fact that $\whPhi F''$ lies in $\Coh^{\leq 1}(X)$, which follows from Lemma \ref{lem:paper14lem5-9analogue}.
\end{proof}

\section{Harder-Narasimhan property of limit tilt stability}\label{sec:HNpropertylimittiltstab}

We establish the Harder-Narasimhan property of $\oZl$-stability on $\Bl$ in this section.  The idea is to decompose the heart $\Bl$ into a torsion quadruple \eqref{eq:paper14eq35} first, and then show that each of the four components satisfies a finiteness property (Proposition \ref{prop:paper14pro1}).

\begin{lem}\label{lem:paper14lem6-1analogue}
\begin{itemize}
\item[(a)] $\Ac_\bullet$ and $\Coh (p)_0$ are both Serre subcategories of $\Bl$ and torsion classes in $\Bl$.
\item[(b)] $\Ac_{\bullet, 1/2}$ is a torsion class in $\Bl$.
\end{itemize}
\end{lem}

\begin{proof}
(a) That $\Ac_\bullet, \Coh (p)_0$ are Serre subcategories of $\Bl$ was already proved in \ref{lem:paper14lem5-4analogue}.  The argument in \cite[Lemma 6.1(b)]{Lo14} shows that, given an ascending chain in $\Bl$
\[
  U_1 \subseteq U_2 \subseteq \cdots \subseteq U_m \subseteq \cdots \subseteq E'
\]
for some fixed $E'$, where the $U_i$  all lie in $\Ac_\bullet$ (resp.\  $\Coh (p)_0$), the $U_i$ must stabilise; this  means that any object in $\Bl$ has a maximal subobject in $\Ac_\bullet$ (resp.\ $\Coh (p)_0$).  Hence  $\Ac_\bullet, \Coh (p)_0$ are both torsion classes in $\Bl$.

(b) The argument in the proof of \cite[Lemma 6.1(b)]{Lo14} carries over, with Lemma \ref{lem:paper14lem5-7analogue} playing the role of \cite[Lemma 5.7]{Lo14}.
\end{proof}

We now define
\begin{align*}
  \mathcal A_{\bullet,1/2,0} &= \langle \mathcal A_{\bullet,1/2}, \scalea{\gyoung(;;+;*,;;0;*)},  \scalea{\gyoung(;;*;*,;+;*;*)},\scalea{\gyoung(;+;*;*,;+;*;*)},\Fc^{l,-}[1]\rangle \notag\\
  &= \langle \Fc^l[1],
  \xymatrix @-2.3pc{
\scalea{\gyoung(;;;,;;;+)} & \scalea{\gyoung(;;;+,;;;+)} & \scalea{\gyoung(;;;*,;;+;*)} & \scalea{\gyoung(;;+;*,;;+;*)} & \scalea{\gyoung(;;*;*,;+;*;*)} &  \scalea{\gyoung(;+;*;*,;+;*;*)}  \\
& \scalea{\gyoung(;;;+,;;;0)} & & \scalea{\gyoung(;;+;*,;;0;*)} & & \\
& \scalea{\gyoung(;;;+,;;;-)} & & & &
}  \rangle. \label{eq49}
\end{align*}
From the computations in \ref{para:phasecomputations}, all the objects in $\Ac_{\bullet,1/2,0}$ have either $\phi \to \tfrac{1}{2}$ or $\phi \to 0$.  (The computation in \ref{para:phasecomputations}(iv) applies to $ \scalea{\gyoung(;+;*;*,;+;*;*)}$ as well.)

\paragraph We  want to show that $\Ac_{\bullet,1/2,0}$ is a torsion class in $\Bl$.  In order to do this, we need to make sure the results on decomposing categories and torsion classes in \cite[Appendix A]{Lo14} all hold for our Weierstra{\ss} threefolds.

\begin{lemsub}\label{lem:paper14lemA1analogue}
\begin{equation}\label{eq:Cohpileq1decomposition}
\Coh (p)_{\leq 1} = \left \langle \Coh^{\leq 1}(X), \vcenter{\vbox{
\xymatrix @-2.3pc{
 \scalea{\gyoung(;;+;*,;;+;*)} \\
 \scalea{\gyoung(;;+;*,;;0;*)} \\
 \scalea{\gyoung(;;+;*,;;-;*)}
}
}} \right\rangle.
\end{equation}
\end{lemsub}

\begin{proof}
Take any $F \in \Coh (p)_{\leq 1}$; we want to show that $F$ lies in the extension closure on the right-hand side of \eqref{eq:Cohpileq1decomposition}.  To this end, we can assume that $F$ has no subsheaves in $\langle \Coh^{\leq 1}(X), \scalea{\gyoung(;;+;*,;;+;*)}\rangle$, i.e.\ we can assume $F$ is a pure 2-dimensional $\whPhi$-WIT$_1$ sheaf.  Suppose $\ch (F) =  {\left( \begin{smallmatrix} 0 & p^\ast \ol{S} & \ol{a}f \\ 0 & \Theta p^\ast \ol{\eta} & \ol{s} \end{smallmatrix}\right)}$.  Then $(p^\ast H_B) \ch_2(F)=H_B\ol{\eta} \leq 0$ by Lemma \ref{lem:LZ2Remark5-17analogue}.  If $H_B \ol{\eta}<0$, then $F$ would lie in $\scalea{\gyoung(;;+;*,;;-;*)}$, and so we can assume $H_B \ol{\eta}=0$.  The  argument  in the second half of the proof of Lemma \ref{lem:paper14lem5-9analogue} now shows $p^\ast \ol{\eta}=0$ in $A^1(X)$, i.e.\ $\whPhi F[1] \in \Coh^{\leq 1}(X) \cap W_{0,\Phi}$.  The argument in the proof of \cite[Lemma A.1]{Lo14} then shows  $F$ must lie in the right-hand side of \eqref{eq:Cohpileq1decomposition}.
\end{proof}

\subparagraph Adding \label{para:paper14appendixanalogues} to the definitions in \ref{para:boxdiagramsdefinition}, let us set
\begin{align*}
 \scalea{\gyoung(;+;*;*,;+;*;*)}  &= \Coh^3(p)_2 \cap W_{0,X}  \\
  \scalea{\gyoung(;+;*;*,;0;*;*)}  &= \{ E \in \Coh^3(p)_2 \cap W_{1,X} : f\ch_{1}(E)=0\} \\
  \scalea{\gyoung(;+;*;*,;-;*;*)}  &= \{ E \in \Coh^3 (p)_2 \cap W_{1,X} : f\ch_{1}(E)<0\}.
\end{align*}
Using Lemma \ref{lem:LZ2Remark5-17analogue}, the following claims all have the same proofs as their counterparts in \cite[Appendix A]{Lo14}:
\begin{itemize}
\item[(i)] \cite[Lemma A.2]{Lo14}: $\Coh^{\leq 2}(X) = \left\langle \Coh (p)_{\leq 1}(X), \scalea{
  \gyoung(;;*;*,;+;*;*)} \right\rangle$.
\item[(ii)]  \cite[Lemma A.3]{Lo14}: $\Coh (X) = \left\langle \Coh^{\leq 2}(X),  \vcenter{\vbox{
  \xymatrix @-2.3pc{
   \scalea{\gyoung(;+;*;*,;+;*;*)} \\
   \scalea{\gyoung(;+;*;*,;0;*;*)} \\
   \scalea{\gyoung(;+;*;*,;-;*;*)}
   }
   }} \right\rangle$.
\item[(iii)] \cite[Lemma A.4]{Lo14}: Every category in the nested sequence

\begin{align*}
&\scalebox{0.5}{
\xymatrix{
\ysize\gyoung(;;;,;;;+)
}
} \label{eq46}
\\
\subset
&\scalebox{0.5}{
\xymatrix @-2.5pc{
\ysize\gyoung(;;;,;;;+) & \ysize\gyoung(;;;+,;;;+)
}
}
\subset
\scalebox{0.5}{
\xymatrix @-2.5pc{
\ysize\gyoung(;;;,;;;+) & \ysize\gyoung(;;;+,;;;+) \\
& \ysize\gyoung(;;;+,;;;0)
}
}
\subset
\scalebox{0.5}{
\xymatrix @-2.5pc{
\ysize\gyoung(;;;,;;;+) & \ysize\gyoung(;;;+,;;;+) \\
& \ysize\gyoung(;;;+,;;;0) \\
& \ysize\gyoung(;;;+,;;;-)
}
}
\subset
\scalebox{0.5}{
\xymatrix @-2.5pc{
\ysize\gyoung(;;;,;;;+) & \ysize\gyoung(;;;+,;;;+) & \ysize\gyoung(;;;*,;;+;*) \\
& \ysize\gyoung(;;;+,;;;0) & \\
& \ysize\gyoung(;;;+,;;;-) &
}
} \notag
\\
\subset
&\scalebox{0.5}{
\xymatrix @-2.5pc{
\ysize\gyoung(;;;,;;;+) & \ysize\gyoung(;;;+,;;;+) & \ysize\gyoung(;;;*,;;+;*) & \ysize\gyoung(;;+;*,;;+;*) \\
& \ysize\gyoung(;;;+,;;;0) & &  \\
& \ysize\gyoung(;;;+,;;;-) & &
}
}
\subset
\scalebox{0.5}{
\xymatrix @-2.5pc{
\ysize\gyoung(;;;,;;;+) & \ysize\gyoung(;;;+,;;;+) & \ysize\gyoung(;;;*,;;+;*) & \ysize\gyoung(;;+;*,;;+;*) \\
& \ysize\gyoung(;;;+,;;;0) & & \ysize\gyoung(;;+;*,;;0;*) \\
& \ysize\gyoung(;;;+,;;;-) & &
}
}
\subset
\scalebox{0.5}{
\xymatrix @-2.5pc{
\ysize\gyoung(;;;,;;;+) & \ysize\gyoung(;;;+,;;;+) & \ysize\gyoung(;;;*,;;+;*) & \ysize\gyoung(;;+;*,;;+;*) \\
& \ysize\gyoung(;;;+,;;;0) & & \ysize\gyoung(;;+;*,;;0;*) \\
& \ysize\gyoung(;;;+,;;;-) & & \ysize\gyoung(;;+;*,;;-;*)
}
}
\subset
\scalebox{0.5}{
\xymatrix @-2.5pc{
\ysize\gyoung(;;;,;;;+) & \ysize\gyoung(;;;+,;;;+) & \ysize\gyoung(;;;*,;;+;*) & \ysize\gyoung(;;+;*,;;+;*) & \ysize\gyoung(;;*;*,;+;*;*) \\
& \ysize\gyoung(;;;+,;;;0) & & \ysize\gyoung(;;+;*,;;0;*) &\\
& \ysize\gyoung(;;;+,;;;-) & & \ysize\gyoung(;;+;*,;;-;*) &
}
} \notag
\\
\subset
&\scalebox{0.5}{
\xymatrix @-2.5pc{
\ysize\gyoung(;;;,;;;+) & \ysize\gyoung(;;;+,;;;+) & \ysize\gyoung(;;;*,;;+;*) & \ysize\gyoung(;;+;*,;;+;*) & \ysize\gyoung(;;*;*,;+;*;*) & \ysize\gyoung(;+;*;*,;+;*;*)\\
& \ysize\gyoung(;;;+,;;;0) & & \ysize\gyoung(;;+;*,;;0;*) & & \\
& \ysize\gyoung(;;;+,;;;-) & & \ysize\gyoung(;;+;*,;;-;*) & &
}
}
\subset
\scalebox{0.5}{
\xymatrix @-2.5pc{
\ysize\gyoung(;;;,;;;+) & \ysize\gyoung(;;;+,;;;+) & \ysize\gyoung(;;;*,;;+;*) & \ysize\gyoung(;;+;*,;;+;*) & \ysize\gyoung(;;*;*,;+;*;*) & \ysize\gyoung(;+;*;*,;+;*;*)\\
& \ysize\gyoung(;;;+,;;;0) & & \ysize\gyoung(;;+;*,;;0;*) & & \ysize\gyoung(;+;*;*,;0;*;*)\\
& \ysize\gyoung(;;;+,;;;-) & & \ysize\gyoung(;;+;*,;;-;*) & &
}
}
\subset
\scalebox{0.5}{
\xymatrix @-2.5pc{
\ysize\gyoung(;;;,;;;+) & \ysize\gyoung(;;;+,;;;+) & \ysize\gyoung(;;;*,;;+;*) & \ysize\gyoung(;;+;*,;;+;*) & \ysize\gyoung(;;*;*,;+;*;*) & \ysize\gyoung(;+;*;*,;+;*;*)\\
& \ysize\gyoung(;;;+,;;;0) & & \ysize\gyoung(;;+;*,;;0;*) & & \ysize\gyoung(;+;*;*,;0;*;*)\\
& \ysize\gyoung(;;;+,;;;-) & & \ysize\gyoung(;;+;*,;;-;*) & & \ysize\gyoung(;+;*;*,;-;*;*)
}
} \notag
\end{align*}
%\end{comment}
\item[(iv)] \cite[Lemma A.5]{Lo14}: Every category in the nested sequence %[un-comment out the following sequence before final edit]

%\begin{comment}
\begin{equation*}
\scalebox{0.5}{\xymatrix @-2.5pc{
\ysize{\gyoung(;;;,;;;+)} & \ysize{\gyoung(;;;+,;;;+)} & \ysize{\gyoung(;;;*,;;+;*)} & \ysize{\gyoung(;;+;*,;;+;*)}  \\
& \ysize{\gyoung(;;;+,;;;0)} & &  \\
& \ysize{\gyoung(;;;+,;;;-)} & &
}
}
\subset
\scalebox{0.5}{\xymatrix @-2.5pc{
\ysize{\gyoung(;;;,;;;+)} & \ysize{\gyoung(;;;+,;;;+)} & \ysize{\gyoung(;;;*,;;+;*)} & \ysize{\gyoung(;;+;*,;;+;*)} & \ysize{\gyoung(;;*;*,;+;*;*)}  \\
& \ysize{\gyoung(;;;+,;;;0)} & &  &  \\
& \ysize{\gyoung(;;;+,;;;-)} & & &
}
}
\subset
\scalebox{0.5}{\xymatrix @-2.5pc{
\ysize{\gyoung(;;;,;;;+)} & \ysize{\gyoung(;;;+,;;;+)} & \ysize{\gyoung(;;;*,;;+;*)} & \ysize{\gyoung(;;+;*,;;+;*)} & \ysize{\gyoung(;;*;*,;+;*;*)} \\
& \ysize{\gyoung(;;;+,;;;0)} & & \ysize{\gyoung(;;+;*,;;0;*)}  \\
& \ysize{\gyoung(;;;+,;;;-)} & &
}
}
\subset
\scalebox{0.5}{\xymatrix @-2.5pc{
\ysize{\gyoung(;;;,;;;+)} & \ysize{\gyoung(;;;+,;;;+)} & \ysize{\gyoung(;;;*,;;+;*)} & \ysize{\gyoung(;;+;*,;;+;*)} & \ysize{\gyoung(;;*;*,;+;*;*)} &  \ysize{\gyoung(;+;*;*,;+;*;*)}  \\
& \ysize{\gyoung(;;;+,;;;0)} & & \ysize{\gyoung(;;+;*,;;0;*)} & & \\
& \ysize{\gyoung(;;;+,;;;-)} & & & &
}
}
\end{equation*}
%\end{comment}
is a torsion class in $\Coh (X)$.
\item[(v)] \cite[Remark A.6]{Lo14}: $W_{0,\whPhi} \subset \Ac_{\bullet,1/2,0}$ and  $\Ac_{\bullet,1/2,0}^\circ \subset W_{1,\whPhi} \cap \Tc^l$.
\end{itemize}
This series of claims gives, with the same proof as \cite[Lemma 6.2]{Lo14}, the following lemma:

\begin{lemsub}\label{lem:paper14lem6-2analogue}
$\Ac_{\bullet,1/2,0}$ is a torsion class in $\Bl$.
\end{lemsub}

\paragraph By Lemmas \ref{lem:paper14lem6-1analogue} and \ref{lem:paper14lem6-2analogue}, we now have a nested sequence of torsion classes in $\Bl$
\[
  \Ac_\bullet \subset \Ac_{\bullet, 1/2} \subset \Ac_{\bullet,1/2,0}
\]
which gives the torsion quadruple in $\Bl$
\begin{equation}\label{eq:paper14eq35}
  \left( \Ac_\bullet,\,\,\, \mathcal A_{\bullet,1/2} \cap \Ac_\bullet^\circ, \,\,\, \mathcal A_{\bullet,1/2,0} \cap  \mathcal A_{\bullet,1/2}^\circ, \,\,\, \mathcal A_{\bullet,1/2,0}^\circ \right).
\end{equation}
That is, every object $E \in \Bl$ has a filtration
\[
  E_0 \subseteq E_1 \subseteq E_2 \subseteq E_3 = E
\]
in $\Bl$ where
\begin{itemize}
\item $E_0 \in \Ac_\bullet$,
\item $E_1/E_0 \in  \mathcal A_{\bullet,1/2} \cap \Ac_\bullet^\circ$,
\item $E_2/E_1 \in \mathcal A_{\bullet,1/2,0} \cap  \mathcal A_{\bullet,1/2}^\circ$
\item $E_3/E_2 \in  \mathcal A_{\bullet,1/2,0}^\circ$.
\end{itemize}
In the remainder of this section, we will show that this filtration in $\Bl$ refines to the Harder-Narasimhan filtration with respect to $\oZl$-stability.

\begin{prop}\label{prop:paper14pro1}
The following finiteness properties hold:
\begin{itemize}
\item[(1)] For $\Ac = \mathcal A_{\bullet,1/2}\cap \Ac_\bullet^\circ$:
\begin{itemize}
\item[(a)] There is no infinite sequence of strict monomorphisms in $\Ac$
\begin{equation}\label{eq25}
  \cdots \hookrightarrow E_n \hookrightarrow \cdots \hookrightarrow E_1 \hookrightarrow E_0
\end{equation}
where $\phi (E_{i+1}) \succ \phi (E_i)$ for all $i$.
\item[(b)] There is no infinite sequence of strict epimorphisms in $\Ac$
\begin{equation}\label{eq26}
  E_0 \twoheadrightarrow E_1 \twoheadrightarrow \cdots \twoheadrightarrow E_n \twoheadrightarrow \cdots.
\end{equation}
\end{itemize}
\item[(2)] For $\Ac = \Ac_{\bullet,1/2,0} \cap \Ac_{\bullet,1/2}^\circ$:
    \begin{itemize}
    \item[(a)] There is no infinite sequence of strict monomorphisms \eqref{eq25} in $\Ac$.
    \item[(b)] There is no infinite sequence of strict epimorphisms \eqref{eq26} in $\Ac$.
    \end{itemize}
\item[(3)] For $\Ac = \Ac_{\bullet,1/2,0}^\circ$:
  \begin{itemize}
  \item[(a)] There is no infinite sequence of strict monomorphisms \eqref{eq25} in $\Ac$.
  \item[(b)] There is no infinite sequence of strict epimorphisms \eqref{eq26} in $\Ac$.
  \end{itemize}
\end{itemize}
\end{prop}

\begin{proof}
Let us start by fixing some notation.  For the proofs of part (a)'s, we will let $G_i$ denote the cokernel of $E_{i+1} \hookrightarrow E_i$, so that we have a $\Bl$-short exact sequence
\begin{equation}\label{eq:Eip1EiGi}
0 \to E_{i+1} \to E_i \to G_i \to 0
\end{equation}
for each $i$.  For the proofs of part (b)'s, we will let $K_i$ denote the kernel of $E_i \twoheadrightarrow E_{i+1}$, so that we have a $\Bl$-short exact sequence
\begin{equation}\label{eq:KiEiEip1}
0 \to K_i \to E_i \to E_{i+1} \to 0
\end{equation}
for each $i$.  Since $f\ch_1 \geq 0$ on $\Bl$, in the proof of either part (a) or  part (b) for (1), (2) and (3) we can always assume $f\ch_1 (E_i)$ is constant, and that $f\ch_1 (G_i)=0$ or $f\ch_1(K_i)=0$ for all $i$.  Note that this further implies $f\ch_1(H^j(G_i))=0=f\ch_1(H^j(K_i))$ for all $i,j$.

(1)(a): The first part of the argument for \cite[Proposition 6.3(1)(a)]{Lo14} applies here, and says that we can assume the $H^{-1}(E_i)$ are constant and the $H^{-1}(G_i)$ are zero.  Hence we have a short exact sequence of sheaves
\begin{equation}\label{eq:paper14Prop6.31a-eq1}
0 \to H^0(E_{i+1}) \to H^0(E_i) \to H^0(G_i) \to 0
\end{equation}
for each $i$.

From the definition of $\Ac_{\bullet,1/2}$ we can see that for any $A \in \Ac_{\bullet,1/2}$, the sheaf $H^0(A)$ lies in the category $\Coh (p)_{\leq 1}$, and so $\ch_1(H^0(A))$ is effective.  Given the short exact sequence \eqref{eq:paper14Prop6.31a-eq1}, it follows that for any fixed ample divisor $\omega$ on $X$ the sequence $\omega^2 \ch_1 (H^0(E_i))$ is a decreasing sequence of nonnegative real numbers that eventually becomes constant, meaning  $\omega^2 \ch_1 (H^0(G_i))=0$ for $i \gg 0$, in which case $H^0(G_i) \in \Coh^{\leq 1}(X)$.  This implies that $\phi (G_i)=\tfrac{1}{2}$, which cannot happen under the assumption $\phi (E_{i+1}) \succ \phi (E_i)$.  Hence an infinite chain of monomorphisms as described in (1)(a) cannot exist.

(1)(b): The idea is to replace the positivity result `$HD\ch_{01} \geq 0$ on $\Ac_{\bullet,1/2}$' in the proof of \cite[Proposition 6.3(1)(b)]{Lo14} by the positivity in Lemma \ref{lem:AG42-157ppp}.

From the long exact sequence of cohomology of \eqref{eq:KiEiEip1}, we see that the $H^0(E_i)$ eventually stabilise, so we can assume they are constant, and the long exact sequence reduces to
\[
0 \to H^{-1}(K_i) \to H^{-1}(E_i)\to H^{-1}(E_{i+1}) \to H^0(K_i) \to 0.
\]
As in the proof of \cite[Proposition 6.3(1)(b)]{Lo14}, by considering $\rank (H^{-1}(E_i))$ we see that $H^{-1}(K_i)$ vanishes for $i \gg 0$, and so we can assume $H^{-1}(K_i)=0$  and $K_i = H^0(K_i)$ for all $i$ by omitting a finite number of terms.  Now, by Lemma \ref{lem:AG42-157ppp} below, the sequence of real numbers $\Theta (p^\ast H_B)\ch_1(E_i)$ is decreasing and must eventually become constant, in which case $\Theta (p^\ast H_B)\ch_1(K_i)=0$.  The same argument as in the proof of Lemma \ref{lem:paper14lem5-9analogue} will then show $\ch_1(K_i)=0$, i.e.\ $K_i \in \Coh^{\leq 1}(X)$.  This implies that all the $H^{-1}(E_i)$ are isomorphic in codimension 1, and so we have the sheaf inclusions
\[
  H^{-1}(E_i) \hookrightarrow H^{-1}(E_{i+1}) \hookrightarrow H^{-1}(E_{i+1})^{\ast \ast}
\]
where $H^{-1}(E_{i+1})^{\ast \ast}$ is independent of $i$.  Hence the $H^{-1}(E_i)$ stabilise, and the $E_i$ themselves stabilise.

(2)(a) Replacing $\ch_{10}$ with $f\ch_1$, the  argument for \cite[Proposition 6.3(2)(a)]{Lo14} gives us the following:
\begin{itemize}
\item The $H^{-1}(E_i)$ stabilise, and so we can assume they are constant.
\item For all $i$, we can assume $H^{-1}(G_i)=0$ and that $G_i=H^0(G_i)$ is a $\whPhi$-WIT$_1$ pure 2-dimensional sheaf.
\end{itemize}
We now have the following exact sequence of sheaves from \eqref{eq:Eip1EiGi}
\[
0 \to H^0(E_{i+1}) \to H^0(E_i) \to H^0(G_i) \to 0
\]
for all $i$.  Applying $\whPhi$, we obtain the short exact sequence of sheaves
\[
0 \to \whPhi^1(H^0(E_{i+1})) \to \whPhi^1 (H^0(E_i)) \to \whPhi^1 (H^0(G_i)) \to 0
\]
in which all the terms lie in $\Coh^{\leq 1}(X)$ as a result of  Lemma \ref{lem:paper14lem6-4analogue}.  Then by Lemma \ref{lem:AG42-156claim1}, we have $(p^\ast H_B)\ch_2(\whPhi^1 (H^0(G_i)))=0$ for $i \gg 0$; from the proof of Lemma \ref{lem:AG42-156claim1}, this further implies $\whPhi^1 (H^0(G_i))$ is a fiber sheaf, i.e.\ $G_i$ itself is a fiber sheaf.  Since we observed above that $G_i$ is a pure 2-dimensional sheaf, this means $G_i$ must be zero, i.e.\ the $E_i$ stabilise.

(2)(b) As in the proof of \cite[Proposition 6.3(2)(b)]{Lo14}, we can assume the $H^0(E_i)$ are constant, and  $K_i = H^0(K_i)$ is a pure 2-dimensional $\whPhi$-WIT$_1$ sheaf for all $i$.  Hence we have the exact sequence of cohomology
\[
 0 \to H^{-1}(E_i) \to H^{-1}(E_{i+1}) \to H^0(K_i) \to 0.
\]
Applying $\whPhi$ then gives the exact sequence of sheaves
\begin{equation}\label{eq:HNpropertyProp2beq1}
0 \to \wh{H^{-1}(E_i)} \to \wh{H^{-1}(E_{i+1})} \to \wh{H^0(K_i)} \to 0
\end{equation}
where $\wh{H^0(K_i)} \in \Coh^{\leq 1}(X)$ by Lemma \ref{lem:paper14lem6-4analogue}.  We claim that $\wh{H^{-1}(E_i)}$ is torsion-free for all $i$.  To see this, fix any $i$ and suppose $\wh{H^{-1}(E_i)}$ has a subsheaf $T$ lying in $\Coh^{\leq 2}(X)$.  Let $T_i$ denote the $\Phi$-WIT$_i$ part of $T$.  Then we have an injection of sheaves $T_0 \hookrightarrow \wh{H^{-1}(E_i)}$ the cokernel of which is $\Phi$-WIT$_0$, and so this injection is taken by $\Phi$ to the injection of sheaves $\wh{T_0} \hookrightarrow H^{-1}(E_i)$.  Note that $f\ch_1(\wh{T_0})=-\ch_0(T_0)=0$.  On the other hand, since $\wh{T_0}$ is a subsheaf of $H^{-1}(E_i)$, which lies in $\Fc^{l,-}$ (because $E_i \in \Ac_{\bullet, 1/2}^\circ$), it follows that $\wh{T_0} \in \Fc^{l,-}$ and  $f\ch_1(\wh{T_0})< 0$ if $\wh{T_0}$ is nonzero.  Thus $T_0$ must vanish, i.e.\ $T=T_1$ lies in $\Coh (p)_{\leq 1} \cap W_{1,\Phi}$.  Then $\wh{T} = \Phi T [1] \in \Coh(p)_{\leq 1} \cap W_{0,\whPhi}$, which is contained in $\Ac_{\bullet, 1/2}$. Now
\begin{align*}
\Hom (T,\wh{H^{-1}(E_i)}) &\cong \Hom (\Phi T, \Phi \wh{H^{-1}(E_i)}) \\
&\cong \Hom (\wh{T}[-1],H^{-1}(E_i)) \\
&\cong \Hom (\wh{T},H^{-1}(E_i)[1]) \\
&=0
\end{align*}
where the last equality holds because  $E_i \in \Ac_{\bullet, 1/2}^\circ$.  Thus $T$ itself must be zero, i.e.\ $\wh{H^{-1}(E_i)}$ is torsion-free.

From the exact sequence \eqref{eq:HNpropertyProp2beq1}, we now have the sheaf inclusions
\[
    \wh{H^{-1}(E_i)} \hookrightarrow \wh{H^{-1}(E_{i+1})}  \hookrightarrow  (\wh{H^{-1}(E_{i+1})})^{\ast\ast}
\]
for all $i$, where $(\wh{H^{-1}(E_{i+1})})^{\ast\ast}$ is independent of $i$.  Hence the $H^{-1}(E_i)$ stabilise, and the $E_i$ themselves stabilise.

(3)(a) By \ref{para:paper14appendixanalogues}(v), all the objects $E_i, G_i$ lie in $W_{1,\whPhi} \cap \Tc^l$.  The proof of \cite[Proposition 6.3(3)(a)]{Lo14} applies once we replace `$HD\ch_{01}$' in the proof by `$\omega^2 \ch_1$', for any ample class $\omega$ on $X$.

(3)(b) The proof of \cite[Proposition 6.3(3)(b)]{Lo14} applies.
\end{proof}

\begin{lem}\label{lem:AG42-157ppp}
For any $A \in \Ac_{\bullet,1/2}$ we have $\Theta (p^\ast H_B) \ch_1(A) \geq 0$.
\end{lem}

\begin{proof}
Take any $A \in \Ac_{\bullet,1/2}$.  From the definition of $\Ac_{\bullet,1/2}$, it is clear that $H^0(A) \in \Coh (p)_{\leq 1}$ and $\ch_1(H^0(A))$ is effective.  Lemma \ref{lem:muwleftcomponent} then gives $\Theta (p^\ast H_B)\ch_1 (H^0(A)) \geq 0$.

Also from the definition of $\Ac_{\bullet, 1/2}$, we know $f\ch_1(H^{-1}(A))=0$, and so $\ch_1 (H^{-1}(A)) = p^\ast \ol{\eta}$ for some $\ol{\eta} \in A^1(B)$ by assumption \eqref{eq:assumptioncohomring}.  Then for $u,v>0$ satisfying \eqref{eq:limitcurve} and for $v \gg 0$, we have
\begin{align*}
  0 &\geq \omega^2 \ch_1 (H^{-1}(A)) \text{ by Lemma \ref{lem:4-1n4-3paper14analogue}(1)(a)} \\
  &= \omega^2 p^\ast \ol{\eta} \\
  &= u(hu+2v)\Theta p^\ast (H_B \ol{\eta}).
\end{align*}
Hence $\Theta (p^\ast H_B)\ch_1(H^{-1}(A)) \leq 0$.  Overall, we have $\Theta (p^\ast H_B)\ch_1(A) \geq 0$.
\end{proof}

\begin{lem}\label{lem:paper14lem6-4analogue}
For any $A \in \Ac_{\bullet,1/2,0}$, if we let $A_1$ denote the $\whPhi$-WIT$_1$ part of $H^0(A)$, then $\wh{A_1} \in \Coh^{\leq 1}(X)$.
\end{lem}

\begin{proof}
The proof of \cite[Lemma 6.4]{Lo14} applies without change.
\end{proof}

\begin{lem}\label{lem:AG42-156claim1}
For any $A \in \Coh^{\leq 1}(X)$, we have $(p^\ast H_B)\ch_2(A) \geq 0$.
\end{lem}

\begin{proof}
Take any $A \in \Coh^{\leq 1}(X)$ and write $\ch_2(A) = \Theta p^\ast \ol{\eta} + \ol{a}f$ where $\ol{\eta}\in A^1(B), \ol{a} \in \mathbb{Q}$.  By Lemma \ref{lem:1dimshch2etaeffective}, $\ol{\eta}$ is an effective divisor class on $B$.  Thus $0 \leq H_B \ol{\eta} = (p^\ast H_B)\ch_2(A)$.
\end{proof}

\paragraph We will now write
\begin{align*}
  \Ac_{1/2} &= \Ac_{\bullet,1/2} \cap \Ac_\bullet^\circ \\
  \Ac_0 &= \Ac_{\bullet,1/2,0} \cap \Ac_{\bullet,1/2}^\circ\\
  \Ac_{-1/2} &= \Ac_{\bullet,1/2,0}^\circ.
\end{align*}
This way, the torsion quadruple \eqref{eq:paper14eq35}  in $\Bl$ can be rewritten as
\begin{equation*}
  (\Ac_\bullet,\, \Ac_{1/2}, \, \Ac_0, \, \Ac_{-1/2}).
\end{equation*}

With the same proofs as their counterparts in \cite{Lo14}, we now have the following results that altogether establish the Harder-Narasimhan property for $\oZl$-stability on $\Bl$.

\begin{lem}\label{lem:paper14lem6-5analogue}
For $i=\tfrac{1}{2}, 0, -\tfrac{1}{2}$, and any $F \in \Ac_i$, we have $\phi (F) \to i$.
\end{lem}

\begin{proof}
The proof of \cite[Lemma 6.5]{Lo14} essentially applies, with the intersection numbers $H^2\ch_{10}, D\ch_{11}$ in that proof replaced with $f\ch_1, (p^\ast H_B)\ch_2$, respectively.

Here, we point out the slight modifications needed for the case of $F \in \Ac_{-1/2}$.  As in the proof of \cite[Lemma 6.5]{Lo14}, when $F \in \Ac_{-1/2}$ we have $F = H^0(F) \in W_{1,\whPhi} \cap \Tc^l$ and $F$ has a filtration in $\Coh (X)$
\[
  M_0 \subseteq M_1 = F
\]
where $M_1/M_0 \in \Tc^{l,0}$, while  $M_0$ is a $\whPhi$-WIT$_1$ pure 2-dimensional sheaf.  Let us write $\ch_2 (M_0) = \Theta p^\ast \ol{\eta} + \ol{a}f$.  Since $M_0$ is $\whPhi$-WIT$_1$, we know  $(p^\ast H_B)\ch_2(M_0)=H_B\ol{\eta} \leq 0$ from Lemma \ref{lem:LZ2Remark5-17analogue}, and that  $p^\ast \ol{\eta}$ is effective.

Suppose $(p^\ast H_B)\ch_2(M_0)=H_B \ol{\eta}=0$.  Then $\ol{\eta}=0$ by Corollary \ref{cor:petaeffective}, in which case $\wh{M_0} \in \Coh^{\leq 1}(X)$.  By the same argument as in the last part of the proof of \cite[Lemma A.1]{Lo14},  this implies $M_0$ itself lies in the extension closure $\langle   \Coh (p)_0, \scalea{\gyoung(;;+;*,;;0;*)}\rangle$, i.e.\ $M_0 \in \Ac_{\bullet, 1/2, 0}$.  Being a subsheaf of $F$, we  know $M_0$ also lies in $\Ac_{\bullet, 1/2,0}^\circ$, forcing $M_0=0$.  Hence we must have $(p^\ast H_B)\ch_2 (M_0) < 0$, i.e.\ $M_0 \in \scalea{\gyoung(;;+;*,;;-;*)}$.  Then from the computations in \ref{para:phasecomputations}  we obtain $\phi (F) \to -\tfrac{1}{2}$.
\end{proof}

\begin{lem}\label{lem:paper14lem6-6analogue}
An object $F \in \Ac_\bullet^\circ$ is $\oZl$-semistable in $\Bl$ iff, for some $i=\tfrac{1}{2}, 0, -\tfrac{1}{2}$, we have:
 \begin{itemize}
 \item $F \in \Ac_i$;
 \item for any strict monomorphism $0 \neq F' \hookrightarrow F$ in $\Ac_i$, we have $\phi (F') \preceq \phi (F)$.
 \end{itemize}
\end{lem}

\begin{proof}
From the torsion quadruple $(\Ac_\bullet,\, \Ac_{1/2}, \, \Ac_0, \, \Ac_{-1/2})$ in $\Bl$ and Lemma \ref{lem:paper14lem6-5analogue}, we know that any $\oZl$-semistable object in $\Ac_\bullet^\circ$ must lie in $\Ac_i$ for $i=\tfrac{1}{2}, 0$ or $-\tfrac{1}{2}$.  The remainder of the proof follows that of  \cite[Lemma 6.6]{Lo14}.
\end{proof}

\begin{thm}\label{thm:paper14thm6-7analogue}
The Harder-Narasimhan property holds for $(\oZw, \Bl)$.  That is, every object $F \in \Bl$ admits a filtration in $\Bl$
\[
  F_0 \subseteq F_1 \subseteq \cdots \subseteq F_n = F
\]
where each subfactor $F_{i+1}/F_i$ is $\oZl$-semistable, and $\phi (F_i/F_{i-1}) \succ \phi (F_{i+1}/F_i)$ for each $i$.
\end{thm}

\begin{proof}
The strategy of the proof is the same as that of \cite[Theorem 6.7]{Lo14}, but we still describe the outline here for the sake of clarity.

By the torsion quadruple $(\Ac_\bullet,\, \Ac_{1/2}, \, \Ac_0, \, \Ac_{-1/2})$, every object $F$ in $\Bl$ has  a filtration
\[
 F_\bullet \subseteq F_1 \subseteq F_2 \subseteq F_3 = F
\]
where $F_\bullet, F_1/F_\bullet, F_2/F_1, F_3/F_2$ lie in $\Ac_\bullet, \Ac_{1/2}, \Ac_0, \Ac_{-1/2}$, respectively.  Using the same argument as in \cite[Theorem 2.29]{Toda1} (which in turn follows the argument of \cite[Proposition 2.4]{StabTC}), along with Lemma \ref{lem:paper14lem6-6analogue} and the finiteness conditions in Proposition \ref{prop:paper14pro1}, one can show that any infinite chain of destabilising quotients or destabilising subobjects using strict morphisms in an $\Ac_i$ must terminate, meaning $F_1/F_\bullet, F_2/F_1, F_3/F_2$ admit HN filtrations in $\Ac_{1/2}, \Ac_0, \Ac_{-1/2}$, respectively.  Concatenating these filtrations then gives a filtration for $F/F_\bullet$
\[
  F'_0 \subseteq F'_1 \subseteq \cdots \subseteq F'_m = F/F_\bullet
\]
in $\Bl$ where $\phi (F_0') \succ \phi (F'_1/F'_0) \succ \cdots \succ \phi (F'_m/F'_{m-1})$ and each $F_i'/F_{i-1}'$ (for $1 \leq i \leq m$) is $\oZl$-semistable.  Lifting this  filtration for $F/F_\bullet$ to a filtration
\begin{equation}\label{eq:fil1}
  F_0 \subseteq F_1 \subseteq \cdots \subseteq F_m = F
\end{equation}
and noting $\phi (F_i/F_{i-1}) = \phi (F'_i/F'_{i-1})$ for all $i$, we have that $\phi (F_0) \succ \phi (F_1/F_0) \succ \cdots \succ \phi (F_m/F_{m-1})$ and that each $F_i/F_{i-1} \cong F'_i/F'_{i-1}$ (for $i \geq 1$) is $\oZl$-semistable.

To finish off, if $\phi (F'_0)=\phi (F_0/F_\bullet)=\tfrac{1}{2}$, then $\phi (F_0)=\tfrac{1}{2}$ and \eqref{eq:fil1} is the HN filtration for $F$.  On the other hand, if $\tfrac{1}{2} \succ \phi (F'_0)$, then the filtration
\[
 F_\bullet \subseteq F_0 \subseteq F_1 \subseteq \cdots \subseteq F_m = F
\]
is the HN filtration of $F$ with respect to $\oZl$-semistability.
\end{proof}

We also have the following analogue of  \cite[Lemma 7.1]{Lo14} with the same proof:

\begin{lem}\label{lem:limvsnonlimtiltstab}
Suppose $a,b>0$ are fixed and satisfy \eqref{eq:constraint1}.  Let $\omega$ be an ample divisor of the form \eqref{eq:boandomega} with $v,u>0$ satisfying \eqref{eq:limitcurve}.  Suppose $E \in D^b(X)$ is such that, for some $v_0 > 0$, we have $E \in \mathcal{B}_\omega$ and $E$ is $\nu_\omega$-(semi)stable for all $v > v_0$.  Then $E \in \Bc^l$ and $E$ is $\oZl$-(semi)stable.
\end{lem}

Lemma \ref{lem:limvsnonlimtiltstab} makes precise a connection between tilt stability and limit tilt stability on elliptic threefolds; informally, one can think of it as saying that tilt stability ``for  $\omega$ sufficiently close to the fiber direction'' implies limit tilt stability.  This is similar to how, on an elliptic surface, Bridgeland stability ``sufficiently close to the fiber direction'' implies Bridgeland stability - see \cite[Theorem 12.3]{Lo20}. Therefore, Lemma \ref{lem:limvsnonlimtiltstab} suggests that limit tilt stability should be closely related to the behaviour of Bridgeland stability conditions under autoequivalences on elliptic threefolds.

\bibliography{refs}{}
\bibliographystyle{plain}

\end{document}